\documentclass[10pt,reqno]{amsart}
\usepackage[scale=0.75, centering, headheight=14pt]{geometry}

\usepackage{mathtools}
\mathtoolsset{showonlyrefs}
\usepackage{graphicx}
\usepackage{fancyhdr}
\usepackage{amssymb,amsmath,enumerate}
\usepackage{caption}
\usepackage{float}
\usepackage{subcaption}
\usepackage{comment,color}
\usepackage{placeins}
\usepackage{amsthm}
\usepackage{thmtools}
\usepackage{bbm}
\newcounter{assump}

\usepackage{stmaryrd} 
\usepackage[colorlinks=true, linkcolor=black, citecolor=black, urlcolor=black]{hyperref}
\usepackage[nameinlink, noabbrev]{cleveref}
\usepackage{tikz}
\DeclareMathOperator*{\argmin}{arg\,min}

\newtheorem{theorem}{Theorem}[section]
\newtheorem{lemma}[theorem]{Lemma}

\newtheorem{remark}[theorem]{Remark}
\newtheorem{corollary}[theorem]{Corollary}

\begin{document}

\title[Fully discrete least-squares splitting for the Monge-Amp\`ere equation]{Fully discrete least-squares splitting scheme for the Monge-Amp\`ere equation:
finite element analysis and convergence}

\author[A. Peruso]{Anna Peruso}

\address{
Institute of Mathematics, \'Ecole Polytechnique F\'ed\'erale de Lausanne, 
1015 Lausanne, Switzerland, 
and
Geneva School of Business Administration (HEG-GEN\`EVE), 
University of Applied Sciences and Arts Western Switzerland (HES-SO), 
1227 Carouge, Switzerland, 
}
\email{anna.peruso@epfl.ch, anna.peruso@hesge.ch}

\date{\today}

\begin{abstract}
The least-squares splitting algorithm for the Monge-Amp\`ere equation has been used successfully in computations for several years, but a convergence theory for fully discrete splitting schemes of this type has remained unavailable. In this work, we introduce and analyze a finite element framework for smooth solutions of the Dirichlet Monge-Amp\`ere equation in two dimensions. The proposed schemes combine a discrete Hessian reconstruction with a local projection onto the determinant constraint. Under a discrete Miranda-Talenti estimate and standard approximation properties of the Hessian reconstruction, we prove local convergence of the iterative scheme and optimal-order convergence of its limit to the exact solution in an \(H^2\)-type norm. We verify the estimates for conforming \(C^1\) schemes, including the Argyris element, and for \(C^0\)-interior penalty and DG schemes of degree at least three; quadratic \(C^0\)-interior penalty and DG schemes are also covered when \(|u|_{H^3(\Omega)}\) is sufficiently small. To the best of our knowledge, these discretizations have not previously been proposed or analyzed for least-squares splitting methods. Numerical experiments confirm the theoretical convergence rates.
\end{abstract}

\maketitle

\section{Introduction}

Fully nonlinear second-order partial differential equations arise in a wide range of applications, including differential geometry, optimal transport, geometric optics, meteorology, and stochastic control \cite{dephilippis,feng}. Among them, the Monge-Amp\`ere equation is one of the prototypical examples. In its elliptic Dirichlet form, on a bounded convex domain \(\Omega\subset \mathbb R^2\), it reads as
\[
    \det D^2 u = f \quad \text{in } \Omega,
    \qquad
    u=g \quad \text{on } \partial\Omega,
\]
where the admissible solution is required to be convex. This convexity requirement selects the elliptic branch of the equation and is essential for uniqueness \cite{caffarelli_cabre,dephilippis}.

The analysis of fully nonlinear equations is substantially more delicate than that of linear or semilinear elliptic problems. Even for smooth data, classical regularity may require additional structural assumptions on the domain, the boundary data, and the solution itself, see \cite{caffarelli_cabre,dephilippis}. From the numerical point of view, several difficulties arise. First, fully nonlinear second-order equations do not generally admit a standard weak formulation obtained by integration by parts. Consequently, methods aimed at approximating smooth solutions must either work directly with Hessians, and therefore with an \(H^2\)-type structure, or introduce suitable discrete Hessians, stabilizations, or auxiliary variables. Second, the nonlinearity has to be handled in a robust way at the discrete level. Third, for equations such as Monge-Amp\`ere, uniqueness is tied to the correct admissible class, namely the class of convex functions.

Nonetheless, over the last decades, several Galerkin-based methods have been developed for smooth solutions of the Dirichlet Monge-Amp\`ere problem. These include vanishing moment methods \cite{neilan}, mixed finite element methods \cite{lakkis,neilan_mixed}, \(C^1\) finite element methods \cite{bohmer,awanouc1,awanouspline}, \(C^0\)-interior penalty methods \cite{brenner,awanou}, discontinuous Galerkin methods \cite{unified}, and least-squares formulations \cite{brennerconvexity,amireh,brenner_convexity2}; see also the surveys \cite{feng,neilan-salgado}. The present work focuses on the nonlinear least-squares splitting method introduced by Dean and Glowinski \cite{dean} and subsequently developed in \cite{caboussat,dimitrios,peruso}. This method introduces an auxiliary matrix variable \(\mathbf P\), interpreted as an approximation of \(D^2u\), and rewrites the Monge-Amp\`ere equation as the constrained problem
\[
    \mathbf P = D^2u\,\text{ in }\Omega,
    \qquad
    \det \mathbf P = f\,\text{ in }\Omega,
    \qquad
    u = g \,\text{ on }\partial\Omega.
\]
The corresponding least-squares formulation, defined in \Cref{sec:setting}, looks for the minimizers of the $L^2$ distance between the set of admissible Hessians and the nonlinear constraint set. The resulting splitting algorithm alternates between a linear variational problem and a pointwise nearest-point projection \cite{dean,glowinski,caboussat}. Its computational efficiency stems from the fact that the nonlinear subproblem is local, while the differential subproblem remains linear.

The method has shown good numerical performance, especially in low-order \(C^0\) implementations \cite{caboussat,peruso}, where the differential subproblem can be reduced to the solution of Poisson-type problems. It has also been extended to optimal transport and optics problems \cite{prins,yadav}, as well as to other classes of fully nonlinear PDEs \cite{pucci,lsjac,lsorthogonal}. However, despite this numerical evidence, a rigorous convergence analysis of the fully discrete splitting iteration has not been made available. A recent analysis in the continuous setting reinterprets the method as an alternating projection algorithm in Sobolev spaces and proves convergence at the continuous level on $\mathbb{T}^2$; see \cite{maxanna}. The present work develops a unified framework for the corresponding fully discrete convergence analysis on general bounded convex polygonal domains. We prove that, under suitable assumptions on the discretization of the linear variational subproblem, the iterative method converges locally and its limit converges to the exact solution in an $H^2$-type norm at optimal rate with respect to the mesh size. We further verify these assumptions for representative $C^1$, $C^0$ and DG finite element methods.

\subsection{Related works}\label{secc:related_works}
There are (at least) two ways to organize the existing literature on Galerkin-type methods for the Monge-Amp\`ere equation. One may classify methods according to the discrete space used, for instance \(C^1\), \(C^0\)-interior penalty, discontinuous Galerkin, or mixed finite element spaces. Alternatively, one may classify them according to the nonlinear strategy used to solve the discrete problem, for instance Newton's method, time-marching methods, fixed-point iterations, or least-squares splitting methods.

A substantial part of the finite element literature relies on Newton-type linearizations. For instance, local well-posedness and convergence estimates for discrete Monge-Amp\`ere problems have been obtained for \(C^1\), \(C^0\)-interior penalty, discontinuous Galerkin, and mixed finite element methods \cite{bohmer,brenner,neilan_mixed,unified}. The convergence of Newton's method for such discretizations has also been studied, for example, in \cite{quadratic_neilan,neilan_mixed}. These results are local: the discrete solution is constructed, and the nonlinear iteration is shown to converge, under the assumption that the initial guess belongs to a mesh-dependent neighbourhood of the exact discrete solution.

Newton's method is not the only possible iterative strategy. \cite{awanou} proposed a time-marching method based on \(C^0\) finite elements, using an explicit Euler discretization of the evolution equation $\partial_t(\Delta u)=\det D^2u - f$. As with Newton's method, convergence is again proved in a local neighborhood whose radius depends on the mesh size. Mesh-dependent local neighbourhoods also appear in the analysis of the vanishing moment method \cite{feng_neilan_mixed,neilan,feng_neilan_2014}, where the relevant fixed-point argument is carried out for the \(\varepsilon\)-regularized fourth-order
problem.

The fact that the convergence neighbourhood shrinks with the mesh size is closely related to the need to control the discrete Hessian in a pointwise norm, as also highlighted in \cite{maxanna}. Since the nonlinear operator depends on \(D^2u\), one needs to ensure that the discrete Hessian remains in the correct elliptic regime. At the discrete level, such a pointwise control is usually obtained from an inverse estimate, which converts an \(H^2\)-type bound into a \(W^{2,\infty}\)-type bound at the cost of a negative power of \(h\). This mechanism is also present in the analysis developed below.

\subsection{Contribution}
The main contribution of this work is the first local convergence theory for fully discrete least-squares splitting schemes for the Dirichlet Monge-Amp\`ere equation, showing that these methods admit the same type of local convergence guarantees as more classical nonlinear iterative methods. We formulate the method in an abstract finite element framework, in terms of a discrete Hessian reconstruction \(H_h\), a stabilization bilinear form \(J_h\), and a pointwise metric projection onto the determinant constraint. This setting includes conforming \(C^1\), \(C^0\)-interior penalty, and discontinuous Galerkin discretizations. To the best of our knowledge, these discretizations have not previously been combined with the Dean-Glowinski least-squares splitting strategy.

Following the continuous analysis in \cite{maxanna}, we interpret the least-squares splitting algorithm as a discrete alternating minimization method. The convergence analysis relies on two assumptions on the reconstruction step: a discrete Miranda-Talenti stability estimate and a consistency estimate for the discrete Hessian reconstruction. Under these assumptions, we prove in \Cref{thm:local_fixed_point} that the fully discrete alternating minimization map is a strict contraction in a mesh-dependent neighbourhood of the exact solution. This yields the existence of a locally unique discrete fixed point and convergence of the iteration. The key estimate, \Cref{lem:transverlality}, is a discrete transversality bound between the linear reconstruction step and the tangent space to the determinant constraint. We then verify the abstract assumptions for conforming \(C^1\), including the Argyris element, and for \(C^0\)-interior penalty, and DG discretizations of polynomial degree at least three. Quadratic \(C^0\)-interior penalty and DG discretizations are also covered provided that the \(H^3\)-seminorm of the exact solution is sufficiently small, see \Cref{rem:valuealpha}. \Cref{thm:local_fixed_point} also yields the reconstructed Hessian error estimates stated in
\Cref{cor:conv_c1,cor:conv_c0,cor:convdg}.

For the original low-order realizations of \cite{caboussat,peruso}, the main remaining obstacle is an appropriate discrete Miranda-Talenti estimate; the same framework may also be applied to other existing or future low-order
reconstructions once the required stability and consistency properties are proven. Finally, as discussed in \Cref{rem:extension}, the method could extend to other smooth uniformly elliptic fully nonlinear operators.

\subsection{Outline}
The remainder of this work is organized as follows. \Cref{sec:setting} introduces the continuous least-squares formulation and splitting algorithm for the Monge-Amp\`ere equation. \Cref{sec:definition} presents the abstract fully discrete framework, while \Cref{sec:conv} establishes its convergence. In \Cref{sec:discreteMT}, we verify the assumptions for \(C^1\), \(C^0\)-interior penalty, and discontinuous Galerkin finite element methods. Numerical results are presented in \Cref{sec:numres}.

\section{Problem setting and algorithm}\label{sec:setting}
\subsection{Notation.}\label{sec:not1}
Henceforth, for an open set \(\Omega\subset \mathbb R^2\), we denote by \(H^s(\Omega)\) and \(W^{s,p}(\Omega)\), \(s\ge 0\), \(1\le p\le\infty\), the standard Sobolev spaces, equipped with their usual norms and seminorms. Vector and matrix-valued versions are understood componentwise.

Let \(n\in\mathbb N\). Matrices in \(\mathbb R^{n\times n}\) are denoted by capital letters, e.g. \(A\), while matrix fields on \(\Omega\) are denoted by bold capital letters, e.g. \(\mathbf A\). We write \(\mathbb S^n\subset\mathbb R^{n\times n}\) for the subspace of symmetric matrices and \(\mathbb S^n_{++}\) for the subset of symmetric positive definite matrices. For \(A,B\in\mathbb R^{n\times n}\), we use the Frobenius product \(A:B := \operatorname{tr}(A^\top B)\), and the associated Frobenius norm \(|A| := \sqrt{A:A}\). We denote by \(L^2(\Omega;\mathbb R^{n\times n})\) the space of square-integrable matrix fields, equipped with the norm
\[
\|\mathbf A\|_{L^2(\Omega)}^2
:= \int_\Omega |\mathbf A(x)|^2\,dx.
\]
Given a subset \(M\subset L^2(\Omega;\mathbb R^{n\times n})\), we denote by \(\Pi_M\) the metric projection onto \(M\) with respect to this \(L^2\)-inner product.

From now on, let \(\Omega\subset\mathbb R^2\) be a convex polygonal domain, and let \(\mathcal T_h\) be a shape-regular, quasi-uniform triangulation of \(\Omega\). For \(K\in\mathcal T_h\), we set \(h_K:=\operatorname{diam}(K)\) and \(h:=\max_{K\in\mathcal T_h}h_K\). For \(r\in\mathbb N_0\), we define the broken Sobolev space
\[
H^r(\mathcal T_h)
:=
\bigl\{v\in L^2(\Omega): v|_K\in H^r(K)
\ \text{for all } K\in\mathcal T_h\bigr\},
\]
equipped with the norm
\begin{equation}\label{eq:brokennorm}
\|v\|_{H^r(\mathcal T_h)}^2
:=
\sum_{K\in\mathcal T_h}\|v\|_{H^r(K)}^2.
\end{equation}
Similarly, we set
\[
\|v\|_{L^\infty(\mathcal T_h)}
:=
\max_{K\in\mathcal T_h}\|v\|_{L^\infty(K)}.
\]
Let \(m\ge 0\) be fixed. We define the broken polynomial space
\begin{equation}\label{eq:broken_pol}
Z_h
:=
\bigl\{v_h\in L^2(\Omega): v_h|_K\in\mathbb P_m(K)
\ \text{for all } K\in\mathcal T_h\bigr\}.
\end{equation}
Elements of \(Z_h\) are identified with their elementwise polynomial representatives. In particular, restrictions and point values are always understood locally on each element. In the present work, we use a nodal norm. For every \(K\in\mathcal T_h\), let \(\{x_{K,j}\}_{j=1}^{N_m}\) be the local nodes obtained from a fixed unisolvent nodal configuration on the reference triangle by affine transformation, where \(N_m=\dim\mathbb P_m(K)\). The nodes are understood locally: if two neighbouring elements share the same physical node, the corresponding local degrees of freedom are still treated as distinct. Thus, for \(v_h\in Z_h\), we write \(v_h(x_{K,j}) := (v_h|_K)(x_{K,j})\). We associate with these nodes positive weights \(w_{K,j}\) satisfying \(w_{K,j}\simeq h_K^2\) uniformly with respect to \(K\), \(j\), and \(h\). For example, if \(m=1\), one may take the vertices of \(K\) and set \(w_{K,j}=|K|/3\); alternative choices are considered in \Cref{sec:numres}. We then define the nodal scalar product
\[
(v_h,z_h)_h
:=
\sum_{K\in\mathcal T_h}\sum_{j=1}^{N_m}
w_{K,j}v_h(x_{K,j})z_h(x_{K,j}),
\qquad v_h,z_h\in Z_h.
\]
The associated norm is defined by \(\|v_h\|_h^2 := (v_h,v_h)_h\). Similarly, we define
\[
\|v_h\|_{h,\infty}
:=
\max_{K\in\mathcal T_h}
\max_{1\le j\le N_m}
|v_h(x_{K,j})|,
\qquad v_h\in Z_h.
\]
We shall use the same notation for any piecewise function whose values at the local nodes \(x_{K,j}\) are well defined. On \(Z_h\), the norms \(\|\cdot\|_h\) and \(\|\cdot\|_{h,\infty}\) are equivalent to \(\|\cdot\|_{L^2(\mathcal T_h)}\) and
\(\|\cdot\|_{L^\infty(\mathcal T_h)}\), respectively, with constants independent of \(h\). Moreover,
\begin{equation}\label{eq:inverseest}
\|v_h\|_{h,\infty}
\le C_{\mathrm{inv}} h^{-1}\|v_h\|_h,
\qquad \forall v_h\in Z_h,
\end{equation}
where \(C_{\mathrm{inv}}>0\) is independent of \(h\), see for instance \cite{brenner_fem}. The definitions of \(H^r(\mathcal T_h)\), \(Z_h\), \(\|\cdot\|_{H^r(\mathcal T_h)}\), \(\|\cdot\|_h\), and \(\|\cdot\|_{h,\infty}\) extend componentwise to vector-valued and
matrix-valued spaces, using the same local nodes and the Euclidean or Frobenius norm at each node. In particular, \(Z_h(\mathbb S^n)\) denotes the space of elementwise polynomial fields with values in \(\mathbb S^n\).

In analogy with the \(L^2\)-metric projection, for a subset \(M\subset Z_h\) we denote by \(\Pi_M^{(h)}\) the metric projection onto \(M\) with respect to the norm \(\|\cdot\|_h\). Finally, if \(L:Z_h\to Z_h\) is a bounded linear operator, we define
\[
\|L\|_{h\to h}
:=
\sup_{\substack{v_h\in Z_h\\ \|v_h\|_h\le 1}}
\|Lv_h\|_h .
\]

\subsection{Model problem and algorithm.}
Let \(\Omega \subset \mathbb{R}^2\) be a convex polygonal domain with boundary \(\partial\Omega\). Assume that \(f \in C^0(\overline{\Omega})\) satisfies \( 0< c_0 \leq f \leq c_1 < \infty\) in \(\Omega\), and let \(g \in H^{3/2}(\partial\Omega)\). The elliptic Dirichlet Monge-Amp\`ere problem is given by: find $u:\Omega\to\mathbb{R}$ such that 
\begin{equation}\label{eq:prob}
\begin{cases}
\det D^2 u = f \quad &\text{in } \Omega,\\
u = g \quad &\text{on } \partial\Omega,
\end{cases}
\end{equation}
where the unknown function \(u\) is required to be convex and \(D^2u\) denotes its Hessian, i.e. \([D^2u]_{ij} = \partial_{x_i x_j} u\). 

We adopt the nonlinear least-squares framework of \cite{dean,caboussat}, which introduces an auxiliary variable. Define \(\mathbf P := D^2u \in L^2(\Omega;\mathbb{S}^2)\). Then \eqref{eq:prob} can be rewritten as
\begin{equation}\label{eq:prob1}
\begin{cases}
\text{\rm det}\:\mathbf{P} = f \quad &\text{in } \Omega, \\
\mathbf{P} = D^2u\quad &\text{in } \Omega, \\
u = g \quad &\text{on } \partial\Omega.\\
\end{cases}
\end{equation}
Since we seek a convex solution to \eqref{eq:prob}, we impose the additional requirement that $\mathbf{P}$ is symmetric and positive definite, hereafter spd. We now introduce the functional spaces and sets used to formulate the solution of \eqref{eq:prob1}:
\begin{align} \label{d:V-g}
    \mathcal{V}_g &:= \{ D^2v \in L^2(\Omega,\mathbb{S}^2) :\, v \in H^{2}(\Omega),\ v|_{\partial\Omega} = g \}
\end{align}
and 
\begin{equation} \label{d:B}
    \mathcal{B} := \{ \mathbf{Q} \in L^2(\Omega,\mathbb{S}^2) :\ \det \mathbf{Q}(x) = f(x)\ \text{a.e. in } \Omega,\ \mathbf{Q}(x)\ \text{spd a.e. in } \Omega  \}.
\end{equation}

It follows that in \eqref{eq:prob1} we seek $\mathbf{P} \in \mathcal{B}$ and $D^2u \in \mathcal{V}_g$.  To determine the pair $(D^2u, \mathbf{P})$, we reformulate \eqref{eq:prob1} as the nonlinear least-squares problem
\begin{equation}\label{eq:leastsq}
(D^2u, \mathbf{P}) 
= \argmin_{D^2v \in \mathcal{V}_g,\; \mathbf{Q} \in \mathcal{B}}
\| D^2v - \mathbf{Q} \|^2_{L^2(\Omega)}.
\end{equation}
As observed in \cite{peruso}, we note that \eqref{eq:leastsq} may admit a solution even in cases where \eqref{eq:prob1} does not, namely when $\mathcal{V}_g \cap \mathcal{B} = \emptyset$. However, whenever \eqref{eq:prob1} has a solution, it also satisfies \eqref{eq:leastsq}, and moreover $\|D^2u - \mathbf{P}\| = 0$. To approximate the solution of \eqref{eq:leastsq}, we employ the splitting algorithm proposed in \cite{dean,caboussat,peruso}, which iteratively decomposes the minimization problem \eqref{eq:leastsq} into two subproblems. Specifically, given an initial guess $\mathbf{P}^0 \in L^2(\Omega,\mathbb{S}^2)$, for each $n \ge 0$ we seek $D^2u^n$ and $\mathbf{P}^{n+1}$ such that:
\begin{subequations}\label{eq:splitting}
\begin{align}
\label{eq:biharmonic} D^2u^{n} = &\argmin_{D^2v\in \mathcal{V}_g}\|D^2v-\mathbf{P}^n\|^2_{L^2(\Omega)},\\
\label{eq:firstmin}\mathbf{P}^{n+1} =& \argmin_{\mathbf{Q}\in \mathcal{B}}\|D^2u^n-\mathbf{Q}\|^2_{L^2(\Omega)}.
\end{align}
\end{subequations}
This approach is an instance of alternating minimization, also known as a block coordinate descent (or Gauss-Seidel-type) scheme, and it decouples the nonlinear constraint from the variational part of the problem. More precisely, the second subproblem \eqref{eq:firstmin} is local in $x\in\Omega$, consisting of the pointwise projection of a symmetric matrix onto the constraint set $\mathcal{B}$, and admits an explicit solution via a Lagrange multiplier argument; see \cite{glowinski}. On the other hand, subproblem \eqref{eq:biharmonic} is linear and the Euler-Lagrange equation reads: find $u\in H^2(\Omega)$, $u=g$ on $\partial\Omega$, such that $a(u,v) = L(v)$ for all $v\in H^2(\Omega)\cap H^1_0(\Omega)$, where
\begin{equation}\label{eq:bilinear}
  a(u,v):= \int_{\Omega}D^2u:D^2v,\quad L(v):=\int_{\Omega}\mathbf{P}:D^2v.  
\end{equation}
The problem is well-posed because $a(\cdot,\cdot)$ defines a norm in $H^2(\Omega)\cap H^1_0(\Omega)$. Its numerical approximation by low-order mixed finite element methods is studied in \cite{caboussat,peruso}. In the implementation of \cite{peruso}, the linear subproblem can be reduced to the solution of two Poisson equations with continuous piecewise linear finite elements, making each iteration computationally inexpensive compared with fully coupled Newton linearizations of the Monge-Ampère equation.

We now reinterpret the splitting scheme \eqref{eq:splitting} as an alternating minimization method. As observed in \cite{maxanna},
\eqref{eq:splitting} is a block coordinate descent scheme for the least-squares functional \eqref{eq:leastsq}. Equivalently, it can be written as
\begin{equation}\label{eq:projj}
D^2u^{n} = \Pi_{\mathcal{V}_g}(\mathbf{P}^n), 
\quad 
\mathbf{P}^{n+1} = \Pi_{\mathcal{B}}(D^2u^n),
\quad 
n \ge 0,
\quad 
\mathbf{P}^0 \text{ given}.
\end{equation}
Here \(\Pi_{\mathcal{V}_g}\) denotes the \(L^2\)-orthogonal projection onto the affine space \(\mathcal{V}_g\), while \(\Pi_{\mathcal{B}}\) denotes the metric projection onto the nonconvex constraint set \(\mathcal{B}\), understood where the latter projection is single-valued. Thus, at least locally, the iteration may be written as
\[
\mathbf{P}^{n+1}=T(\mathbf{P}^n),
\qquad
T:=\Pi_{\mathcal{B}}\circ \Pi_{\mathcal{V}_g}.
\]

At the continuous level, local convergence of this alternating projections algorithm was proved in \cite{maxanna} on \(\mathbb T^2\), under suitable regularity and nondegeneracy assumptions on the exact solution. The purpose of the present work is to develop the corresponding fully discrete convergence theory for bounded convex polygonal domains in $\mathbb{R}^2$. This requires additional ingredients which are absent in the continuous analysis, namely a discrete Hessian operator, a discrete norm for the local nonlinear projection, and stability estimates for the Hessian reconstruction step.

\section{Definition of the discrete scheme}\label{sec:definition}
In this section we aim to build a discrete version of the alternating projection algorithm in \eqref{eq:splitting}. We also state the set of assumptions that are sufficient to prove the convergence of this scheme in \Cref{sec:conv}. To make the exposition easier to follow we treat the case $g\equiv 0$; the case of nonhomogeneous Dirichlet boundary condition follows accordingly, for both strong and weak treatment of the Dirichlet boundary condition.

Let $m\geq 0$ be the degree of polynomial in $Z_h$, as defined in \eqref{eq:broken_pol}. We assume in the following that there exists a smooth solution $u\in W^{3,\infty}(\Omega)\cap H^{m+3}(\Omega)$, to \eqref{eq:prob} and that $D^2u$ is uniformly elliptic, i.e. there exist $\nu_1\geq \nu_2 >0$ such that 
\begin{equation}\tag{A.1}\label{eq:elliptic}
    \nu_2|\xi|^2\leq \xi^T D^2u(x)\xi\leq \nu_1|\xi|^2 \quad \forall\xi\in \mathbb{R}^2\quad\forall x\in\Omega \quad \text{and}\quad D^2u\in W^{1,\infty}(\Omega,\mathbb{S}^2)\cap H^{m+1}(\Omega,\mathbb{S}^2).
\end{equation}
Now, we consider a possible discretization of this scheme. Specifically, the goal is to define a discrete scheme where \eqref{eq:firstmin} is solved pointwise on some points of a given mesh $\mathcal{T}_h$. The construction is thus centered around this requirement. Indeed, the benefit of the splitting \eqref{eq:biharmonic}-\eqref{eq:firstmin} is to have a nonlinear but local problem and one linear but variational problem and the numerical scheme should preserve this property. 

With this goal in mind, we define the discrete determinant constraint set
\begin{equation}
\mathcal{B}_h
:=
\left\{
\mathbf{Q}_h\in Z_h(\mathbb{S}^2):\,
\mathbf{Q}_h(x_{K,j})\in \mathbb{S}^2_{++},\,
\det\mathbf{Q}_h(x_{K,j})=f(x_{K,j})
\,\,\forall K\in\mathcal{T}_h,\,
j=1,\dots,N_m
\right\},
\end{equation}
and we denote by
\begin{equation}
    \Pi_{\mathcal{B}_h}^{(h)}:Z_h(\mathbb{S}^2)\rightarrow \mathcal{B}_h
\end{equation}
the metric projection onto $\mathcal{B}_h$ with respect to the norm $\|\cdot\|_h$. Since \(\mathcal B_h\) is nonconvex, the metric projection
\(\Pi_{\mathcal B_h}^{(h)}\) is initially understood as set-valued. \Cref{thm:smooth_projection} shows that it is single-valued in a neighbourhood of
\(\widehat{\mathbf P}_h\). Given that the norm \(\|\cdot\|_h\) is defined only through the local nodal values,
the projection is completely local. Indeed, for a given
\(\mathbf{P}_h\in Z_h(\mathbb{S}^2)\), the value of
\(\Pi_{\mathcal{B}_h}^{(h)}(\mathbf{P}_h)\) at a node \(x_{K,j}\) is obtained by
projecting the matrix \(\mathbf{P}_h(x_{K,j})\) onto the set
\begin{equation}\label{eq:Ma_mani}
    \mathcal{M}_a:=\left\{
    M\in\mathbb{S}^2_{++}:\det M=a
    \right\},
\end{equation}
where $a=f(x_{K,j})$. The projected finite element field is then reconstructed from these nodal values. The local nonlinear minimization problem can be solved, for instance, via the algorithm presented in \cite{glowinski}. Note that the projection onto $\mathcal{B}_h$ with respect to the $L^2$ norm would not have given the same local property.

We first show that the discrete projection \(\Pi_{\mathcal{B}_h}^{(h)}\) is smooth in a neighbourhood of the nodal
interpolant of the exact Hessian \(D^2u\). We then describe the discrete approximation of \eqref{eq:biharmonic}. We denote by \(I_h^Z:C^0(\overline{\Omega})\to Z_h\) the local nodal interpolation operator, defined by
\[
    (I_h^Z v)(x_{K,j})=v(x_{K,j}),
    \qquad K\in\mathcal{T}_h,\quad j=1,\dots,N_m.
\]
The operator $I_h^Z$ is extended componentwise to vector-valued and matrix-valued functions. We define
$$    \widehat{\mathbf{P}}_h := I_h^Z(D^2u),$$
i.e. the interpolant of the exact Hessian. By definition and \eqref{eq:elliptic}, $\widehat{\mathbf{P}}_h$ is symmetric positive definite at every node. Moreover, since $D^2u\in W^{1,\infty}(\Omega)$, standard interpolation estimates \cite{brenner_fem} yield $\|I_h^Z(D^2u)-D^2u\|_{L^\infty(\Omega)} \le Ch\,|D^2u|_{W^{1,\infty}(\Omega)}$. Therefore, for sufficiently small $h$, $\widehat{\mathbf{P}}_h$ remains symmetric positive definite throughout each element.
\begin{theorem}\label{thm:smooth_projection}
Assume \eqref{eq:elliptic} holds. Then, there exist $0 < \tau_0 \leq \nu_2/2$ and $C_\Pi$, both independent of $h$, such that, for every
$0<\tau\leq \tau_0$, the projection $\Pi_{\mathcal{B}_h}^{(h)}$ is single-valued and $C^\infty$ on 
\[
\mathcal{U}_\tau^{\infty}(\widehat{\mathbf{P}}_h)
:=
\left\{
\mathbf{Q}_h\in Z_h(\mathbb{S}^2):
\|\mathbf{Q}_h-\widehat{\mathbf{P}}_h\|_{h,\infty}<\tau
\right\};
\]
every $\mathbf{Q}_h\in \mathcal{U}_\tau^{\infty}(\widehat{\mathbf{P}}_h)$ is symmetric positive definite at each local node and, for sufficiently small $h$, throughout each element. Moreover, for every
$\mathbf{A}_h,\mathbf{B}_h\in \mathcal{U}_\tau^{\infty}(\widehat{\mathbf{P}}_h)$:
\begin{align}
\label{eq:lip}
\|\Pi_{\mathcal{B}_h}^{(h)}(\mathbf{A}_h)
-
\Pi_{\mathcal{B}_h}^{(h)}(\mathbf{B}_h)\|_h
&\leq
C_\Pi
\|\mathbf{A}_h-\mathbf{B}_h\|_h,
\\
\label{eq:difflip}
\|D\Pi_{\mathcal{B}_h}^{(h)}(\mathbf{A}_h)
-
D\Pi_{\mathcal{B}_h}^{(h)}(\mathbf{B}_h)\|_{h\to h}
&\leq
C_\Pi
\|\mathbf{A}_h-\mathbf{B}_h\|_{h,\infty}.
\end{align}
Finally, $D \Pi^{(h)}_{\mathcal{B}_h}(\widehat{\mathbf{P}}_h)[\cdot]=\Pi^{(h)}_{T_{\mathcal{B}_h}(\widehat{\mathbf{P}}_h)}(\cdot)$, where $T_{\mathcal{B}_h}(\widehat{\mathbf{P}}_h)$ is the tangent space to $\mathcal{B}_h$ at $\widehat{\mathbf{P}}_h$, namely: 
\[
T_{\mathcal{B}_h}(\widehat{\mathbf{P}}_h)
:=
\left\{
\mathbf{W}_h\in Z_h(\mathbb{S}^2),\,\,\,
\operatorname{cof}(\widehat{\mathbf{P}}_h(x_{K,j}))
:
\mathbf{W}_h(x_{K,j})
=
0
\quad
\forall K\in\mathcal{T}_h,\ j=1,\dots,N_m
\right\},
\] 
where $\operatorname{cof}(\mathbf{P}) := \det(\mathbf{P})\mathbf{P}^{-T}$.

\end{theorem}

\begin{proof}

Let $\mathcal{M}_a$ be defined by \eqref{eq:Ma_mani}. The set \(\mathcal M_1\) is a smooth embedded hypersurface of \(\mathbb S^2\), since \(D(\det)(N)[W]=\operatorname{cof}(N):W\) and \(\operatorname{cof}(N)\neq0\) for every \(N\in\mathcal M_1\). At each local node, let $a:=f(x_{K,j})$ and $N_{K,j}:=a^{-1/2}\widehat{\mathbf P}_h(x_{K,j})$. Then \(\det N_{K,j}=1\), and \eqref{eq:elliptic} together with \(c_0\leq a\leq c_1\) implies that all normalized nodal values $N_{K,j}$ belong to the compact set
\[
\mathcal K_0
:=
\left\{
M\in\mathcal M_1:
\frac{\nu_2}{\sqrt{c_1}}I
\leq M
\leq
\frac{\nu_1}{\sqrt{c_0}}I
\right\}.
\]
Let \(\mathcal E(\mathcal M_1)\subset\mathbb S^2\) be the open set on which the nearest-point projection \(\pi_1:\mathcal E(\mathcal M_1)\to\mathcal M_1\) is uniquely defined. Since \(\mathcal M_1\) is \(C^\infty\), \(\mathcal M_1\subset\mathcal E(\mathcal M_1)\) and \(\pi_1\) is \(C^\infty\); see \cite{Leobacher}. Morever, there exists \(r>0\)
such that
\[
\left\{
X\in\mathbb S^2:
\operatorname{dist}(X,\mathcal K_0)\leq r
\right\}
\subset\mathcal E(\mathcal M_1),
\]
and, by compactness, there exists \(C>0\) such that
\[
\|D\pi_1(X)\|+\|D^2\pi_1(X)\|\leq C
\]
whenever \(\operatorname{dist}(X,\mathcal K_0)\leq r\).
Since \(\mathcal M_a=\sqrt a\,\mathcal M_1\), the corresponding metric projections satisfy
\[
\pi_a(X)
=
\sqrt a\,\pi_1\left(\frac{X}{\sqrt a}\right),\quad D\pi_a(X)
=
D\pi_1\left(\frac{X}{\sqrt a}\right),
\quad
D^2\pi_a(X)
=
\frac{1}{\sqrt a}
D^2\pi_1\left(\frac{X}{\sqrt a}\right),
\]
see also \cite{glowinski}. Since \(a\geq c_0\), these derivatives are uniformly bounded for \(a\in[c_0,c_1]\). Set \(\tau_0:=\min\{\frac{\nu_2}{2},\sqrt{c_0}\,r\}\). If \(|X-\widehat{\mathbf P}_h(x_{K,j})|<\tau_0\), then \(\operatorname{dist}(X/\sqrt a,\mathcal K_0)\leq r\). Hence \(\pi_a(X)\) is uniquely defined and smooth, with bounds independent of \(h\), \(K\), and \(j\). 

As previously discussed, the projection is computed pointwise:
\[
\bigl(\Pi_{\mathcal B_h}^{(h)}\mathbf Q_h\bigr)(x_{K,j})
=
\pi_{f(x_{K,j})}
\bigl(\mathbf Q_h(x_{K,j})\bigr).
\]
This proves its single-valuedness and smoothness. Moreover, if \(\mathbf Q_h\in\mathcal U_\tau^\infty(\widehat{\mathbf P}_h)\), then
\[
\xi^T\mathbf Q_h(x_{K,j})\xi
\ge
\left(
\nu_2-
|\mathbf Q_h(x_{K,j})
-\widehat{\mathbf P}_h(x_{K,j})|
\right)|\xi|^2
\ge \frac{\nu_2}{2}|\xi|^2,
\]
so \(\mathbf Q_h\) is positive definite at every local node. Moreover, decreasing $\tau_0$ if necessary, since \(\widehat{\mathbf P}_h\) is positive definite on each element for sufficiently small \(h\) and $\|\cdot\|_{h,\infty}$ and $\|\cdot\|_{L^\infty(\mathcal{T}_h)}$ are equivalent, \(\mathbf Q_h\in\mathcal U_\tau^\infty(\widehat{\mathbf P}_h)\) is also positive definite on each element. 
The uniform bound on \(D\pi_a\) yields \eqref{eq:lip}. Moreover, let \(\mathbf{A}_h,\mathbf{B}_h\in \mathcal{U}_\tau^{\infty}(\widehat{\mathbf{P}}_h)\), then, for every \(\mathbf{Z}_h\in Z_h(\mathbb{S}^2)\),
\[
\begin{aligned}
\left|
\left[
(D\Pi^{(h)}_{\mathcal{B}_h}(\mathbf{A}_h)-D\Pi^{(h)}_{\mathcal{B}_h}(\mathbf{B}_h))[\mathbf{Z}_h]
\right](x_{K,j})
\right|\le&
C_\Pi |\mathbf{A}_h(x_{K,j})-\mathbf{B}_h(x_{K,j})|\,|\mathbf{Z}_h(x_{K,j})|
\\
\le&
C_\Pi \|\mathbf{A}_h-\mathbf{B}_h\|_{h,\infty}|\mathbf{Z}_h(x_{K,j})|,
\end{aligned}
\]
which yields \eqref{eq:difflip}. Finally, define the tangent space \(T_{\mathcal M_a}(M):=\{W\in\mathbb S^2:\operatorname{cof}(M):W=0\}\). The differential of the nearest-point projection at \(M\in\mathcal M_a\) is the Frobenius-orthogonal projection onto \(T_{\mathcal M_a}(M)\); see for instance \cite{johnlee,Leobacher}. Applying this identity nodewise gives \(D\Pi_{\mathcal B_h}^{(h)} (\widehat{\mathbf P}_h)=\Pi_{T_{\mathcal B_h}(\widehat{\mathbf P}_h)}^{(h)}\).
\end{proof}

We now turn to the discrete approximation of \eqref{eq:biharmonic}. Let \(X_h\) be the scalar finite element space in which the approximation of \(u\) is sought, e.g. standard Lagrange finite elements. If the homogeneous Dirichlet condition is imposed strongly, then \(X_h\) is the corresponding subspace satisfying this condition. Otherwise, \(X_h\) is the full scalar space, and the boundary condition is enforced weakly through the bilinear form. Concrete examples of $X_h$ will be presented in \Cref{sec:discreteMT}.

Let
\[
H_h:X_h\to Z_h(\mathbb S^2)
\]
be a linear discrete Hessian operator, and let
\[
J_h:X_h\times X_h\to\mathbb R
\]
be a symmetric positive semidefinite stabilization bilinear form, e.g. a jump penalization term. It includes boundary penalty terms whenever required. For \(\sigma>0\), define
\[
a_h(w_h,v_h)
:=
(H_hw_h,H_hv_h)_h+\sigma J_h(w_h,v_h),
\qquad
w_h,v_h\in X_h.
\]
Given \(\mathbf P_h\in Z_h(\mathbb S^2)\), we seek \(u_h=u_h(\mathbf P_h)\in X_h\) satisfying
\begin{equation}
\label{eq:bihC0IP}
a_h(u_h,v_h)
=
(\mathbf P_h,H_hv_h)_h,
\qquad
\forall v_h\in X_h.
\end{equation}
The associated discrete Hessian is then defined by
\[
R_h\mathbf P_h:=H_hu_h(\mathbf P_h).
\]
We do not require \(a_h\) to be coercive on \(X_h\), but only that \(R_h\) be well defined. This follows from the fact that
\[
z_h\in\ker a_h
\quad\Longrightarrow\quad
H_hz_h=0,
\]
which holds since \(J_h\) is symmetric positive semidefinite. Indeed, if \(u_h^1\) and \(u_h^2\) satisfy \eqref{eq:bihC0IP}, then \(u_h^1-u_h^2\in\ker a_h\), so that \(H_hu_h^1=H_hu_h^2\). Hence \(R_h:Z_h(\mathbb S^2)\to Z_h(\mathbb S^2)\) is well defined.

With this notation, the discrete alternating minimization scheme is defined as follows. Given \(\mathbf P_h^0\in Z_h(\mathbb S^2)\), for \(n\geq 0\) choose \((u_h^n,\mathbf P_h^{n+1})\in X_h\times\mathcal B_h\) solving
\begin{subequations}\label{eq:splitting_disc}
    \begin{align}
    \label{eq:biharmonic_disc}u_h^n
    \in
    &\argmin_{v_h\in X_h}
    \left\{
        \|H_hv_h-\mathbf P_h^n\|_h^2
        +\sigma J_h(v_h,v_h)
    \right\},\\
    \label{eq:firstmin_disc}\mathbf P_h^{n+1}
    =&\argmin_{\mathbf{Q}_h\in\mathcal{B}_h}\|\mathbf{Q}_h - H_hu_h^n\|_h^2.     
    \end{align}
\end{subequations}

In order to state the main result, \Cref{thm:local_fixed_point}, we impose two assumptions on the discrete linear map \(R_h\). The first is a stability assumption of Miranda-Talenti type. Recall that, at the continuous level, the Miranda-Talenti inequality \cite{miranda,talenti} gives
\[
    |u|_{H^2(\Omega)}
    \leq
    \|\Delta u\|_{L^2(\Omega)}
    \qquad
    \forall u\in H^2(\Omega)\cap H^1_0(\Omega).
\]
We assume the following discrete analogue: there exists a constant \(C>0\), independent of \(h\), such that
\begin{equation}\label{eq:MTstab}
\tag{A.2}
    \|H_hv_h\|_{L^2(\mathcal T_h)}
    \le
    \|\operatorname{tr}(H_hv_h)\|_{L^2(\mathcal T_h)}
    +
    C J_h(v_h,v_h)^{1/2}
    \qquad
    \forall v_h\in X_h .
\end{equation}
Here the stabilization term compensates for the fact that the discrete Hessian does not necessarily satisfy the exact continuous Miranda-Talenti estimate, when, for instance, $X_h\not \subset H^2(\Omega)$. The second assumption is a consistency requirement at \(\widehat{\mathbf P}_h\). More precisely, we assume that there exist
constants \(C_R>0\) and \(\alpha>0\), independent of \(h\), such that
\begin{equation}\label{eq:boundhR_h}
\tag{A.3}
    \rho_h
    :=
    \|R_h(\widehat{\mathbf P}_h)-\widehat{\mathbf P}_h\|_h
    \leq
    C_R h^{1+\alpha}.
\end{equation}
In other words, \(\widehat{\mathbf P}_h\) is required to be an approximate fixed point of the discrete linear step \(R_h\). This is the discrete analogue of the continuous identity \(\Pi_{\mathcal V_g}(\mathbf P)=\mathbf P\),
which holds for the exact Hessian \(\mathbf P=D^2u\), since \(\mathbf P\in\mathcal{V}_g\). At the discrete level this identity is not expected to hold exactly, because \(\widehat{\mathbf P}_h\) need not belong exactly to the discrete range of \(R_h\). Assumption \eqref{eq:boundhR_h} therefore quantifies the consistency defect.

\section{Convergence of the discrete scheme}\label{sec:conv}
In this section we prove the local convergence of the discrete alternating minimization scheme \eqref{eq:biharmonic_disc}-\eqref{eq:firstmin_disc}. The argument is based on viewing the scheme as the fixed point iteration
\[
    \mathbf P_h^{n+1}=T_h(\mathbf P_h^n),
    \qquad
    T_h:=\Pi_{\mathcal B_h}^{(h)}\circ R_h.
\]
Thus the proof reduces to showing that \(T_h\) is a contraction in a suitable neighbourhood of the discrete approximation \(\widehat{\mathbf P}_h\) of the exact Hessian. There are two separate points to control. The linear step \(R_h\) is handled
through the discrete Miranda-Talenti stability assumption \eqref{eq:MTstab}, while the nonlinear step \(\Pi_{\mathcal B_h}^{(h)}\) is controlled by the smoothness result of \Cref{thm:smooth_projection}, already proven in \Cref{sec:definition}. The consistency assumption \eqref{eq:boundhR_h} ensures that \(\widehat{\mathbf P}_h\) is an approximate fixed point of the linear step. Combining these facts, we prove that \(T_h\) maps a ball of radius \(O(h)\) around \(\widehat{\mathbf P}_h\) into itself and is a strict contraction there.

We now state the resulting local fixed-point theorem.

\begin{theorem}\label{thm:local_fixed_point}
Assume \eqref{eq:elliptic}, \eqref{eq:MTstab} and \eqref{eq:boundhR_h} hold true. Let $\widehat{\mathbf{P}}_h:=I_h^{Z}(D^2u)$. Then, there exist $\kappa_0>0$ and $h_0>0$ such that, $\forall h\leq h_0$, \(T_h\) is a contraction on $\mathcal{U}_{\kappa_0}^h(\widehat{\mathbf{P}}_h):=\{
\mathbf{Q}_h\in Z_h(\mathbb{S}^2):
\|\mathbf{Q}_h-\widehat{\mathbf{P}}_h\|_h\leq \kappa_0 h\}$ and has a locally unique fixed point \(\mathbf{P}_h^\star\in\mathcal{U}_{\kappa_0}^h(\widehat{\mathbf{P}}_h)\). In particular, there exists $q\in (0,1)$, independent of $h$, such that, for every initial datum \(\mathbf{P}_h^0\in \mathcal{U}_{\kappa_0}^h(\widehat{\mathbf{P}}_h)\), 
\begin{equation}
    \|\mathbf{P}_h^n-\mathbf{P}_h^\star\|_h
    \leq
    q^n
    \|\mathbf{P}_h^0-\mathbf{P}_h^\star\|_h,\quad n\geq 0.
\end{equation}
Hence, the iteration \(\mathbf{P}_h^{n+1}=T_h(\mathbf{P}_h^n)\) converges to \(\mathbf{P}_h^\star\). Moreover, there exists $C>0$, independent of $h$, such that 
\begin{equation}\label{eq:errP}
    \|\mathbf{P}_h^\star-\widehat{\mathbf{P}}_h\|_h    \leq  C h^{1+\alpha}.
\end{equation}
\end{theorem}

Before proving \Cref{thm:local_fixed_point}, we isolate two estimates needed in the fixed-point argument. First, in \Cref{lemma:TMimpliesCordes}, we show that the discrete Miranda-Talenti estimate \eqref{eq:MTstab} implies a Cordes-type stability estimate in the discrete norm \(\|\cdot\|_h\). This provides the required control on the linear reconstruction step. Second, in \Cref{lem:transverlality}, we prove a transversality estimate near \(\widehat{\mathbf P}_h\), which is the mechanism responsible for the contraction of the full map $T_h$.

\begin{lemma}\label{lemma:TMimpliesCordes}
    Assume \eqref{eq:elliptic} and \eqref{eq:MTstab}. Then, there exist $h_0>0$ and $C_D>0$, independent of $h$, such that, $\forall h\leq h_0$,
    \begin{equation}
          \|H_hv_h\|_h^2
    \le
        C_D\big(\|\operatorname{cof}(\widehat{\mathbf{P}}_h):H_hv_h\|_h^2
        + J_h(v_h,v_h)\big)
    \qquad
    \forall v_h\in X_h .
    \end{equation}
\end{lemma}

\begin{proof}
Set \(\mathbf{Q}_h:=H_hv_h\) and \(\mathbf{C}_h:=\operatorname{cof}(\widehat{\mathbf{P}}_h)\). By \eqref{eq:elliptic} and standard interpolation estimates in $L^\infty$ \cite{brenner_fem}, for \(h\) small enough, \(\mathbf{C}_h\) is uniformly elliptic and uniformly bounded. Therefore, the Cordes estimate holds uniformly (see, for instance, \cite{campanato,smearssuli}): there exist \(\varepsilon\in(0,1)\) and a uniformly bounded scalar function \(\gamma_h:=\frac{\operatorname{tr}\mathbf{C}_h}{|\mathbf{C}_h|^2}\) such that, pointwise on each element,
\[
\left|\operatorname{tr}\mathbf{Q}_h-\gamma_h \mathbf{C}_h:\mathbf{Q}_h\right|
\le \sqrt{1-\varepsilon}\,|\mathbf{Q}_h|.
\]
Using \eqref{eq:MTstab}, we obtain
\[
\|\mathbf{Q}_h\|_{L^2(\mathcal T_h)}
\le
C\|\mathbf{C}_h:\mathbf{Q}_h\|_{L^2(\mathcal T_h)}
+\sqrt{1-\varepsilon}\|\mathbf{Q}_h\|_{L^2(\mathcal T_h)}
+CJ_h(v_h,v_h)^{1/2},
\]
with $C>0$. Absorbing the second term gives
\begin{equation}\label{eq:cordes_interm}
    \|\mathbf{Q}_h\|_{L^2(\mathcal T_h)}
\lesssim \|\mathbf{C}_h:\mathbf{Q}_h\|_{L^2(\mathcal T_h)}
+J_h(v_h,v_h)^{1/2}.
\end{equation}
It remains to replace the norm in \(L^2(\mathcal{T}_h)\) of \(\mathbf{C}_h:\mathbf{Q}_h\) by the nodal norm. We use the
standard local stability of nodal interpolation on fixed polynomial spaces (see for instance \cite[Chapter 4.4]{brenner_fem}):
\begin{equation}\label{eq:stab_interp}
\|I_h^Z r\|_{L^2(K)}\le C\|r\|_{L^2(K)}
\qquad \forall r\in \mathbb{P}_{2m}(K),    
\end{equation}
with \(C\) independent of \(h\). For each element \(K\), choose \(x_K\in K\) and set \(C_K:=\mathbf{C}_h(x_K)\). Since
\(C_K:\mathbf{Q}_h\in \mathbb{P}_m(K)\), we have \(I_h^Z(C_K:\mathbf{Q}_h)=C_K:\mathbf{Q}_h\), hence
\[
(I-I_h^Z)(\mathbf{C}_h:\mathbf{Q}_h)
=
(I-I_h^Z)\big((\mathbf{C}_h-C_K):\mathbf{Q}_h\big)
\quad\text{on }K.
\]
By \eqref{eq:stab_interp} and the linearity of the cofactor map in 2D,
\begin{align}
\|(I-I_h^Z)(\mathbf{C}_h:\mathbf{Q}_h)\|_{L^2(K)}
\lesssim &
\|\mathbf{C}_h-C_K\|_{L^\infty(K)}\|\mathbf{Q}_h\|_{L^2(K)}\\
\lesssim&h_K\|\nabla \mathbf{C}_h\|_{L^\infty(K)}\|\mathbf{Q}_h\|_{L^2(K)}\\
\lesssim &h_K\|u\|_{W^{3,\infty}(\Omega)}\|\mathbf{Q}_h\|_{L^2(K)}
\end{align}
Consequently,
\begin{equation}\label{eq:boundd}
\|\mathbf{C}_h:\mathbf{Q}_h\|_{L^2(\mathcal T_h)}
\lesssim
\|\mathbf{C}_h:\mathbf{Q}_h\|_h
+h\|\mathbf{Q}_h\|_{L^2(\mathcal T_h)}.    
\end{equation}
Combining \eqref{eq:cordes_interm} and \eqref{eq:boundd}, and absorbing the last term for \(h\) small enough, yields the desired result.
\end{proof}

We next prove that the range of $R_h$ is transversal to the tangent space of $\mathcal{B}_h$ at $\widehat{\mathbf{P}}_h$. This is the discrete analogue of the corresponding transversality result established in \cite{maxanna} for the continuous setting by means of elliptic regularity results.

\begin{lemma}\label{lem:transverlality}
Let the assumptions of \Cref{lemma:TMimpliesCordes} hold. Then, for any $h>0$, $\|R_h\|_{h\to h}\leq 1$. Moreover, there exist $h_0>0$ and a constant \(q_0\in(0,1)\), independent of \(h\), such that, $\forall h\leq h_0$:
\[
    \left\|
        \Pi_{T_{\mathcal{B}_h}(\widehat{\mathbf{P}}_h)}^{(h)} \circ R_h
    \right\|_{h\to h}
    \le q_0 .
\]
\end{lemma}

\begin{proof}
Set \(\mathbf{C}_h := \operatorname{cof}(\widehat{\mathbf{P}}_h)\). Let \(\mathbf{W}_h\in Z_h(\mathbb{S}^2)\) be arbitrary. Let \(u_h\in X_h\) solve \eqref{eq:bihC0IP} with right-hand side $\mathbf{W}_h$. Set \(\mathbf{Q}_h := R_h\mathbf{W}_h = H_hu_h\). Taking \(v_h=u_h\) in
\eqref{eq:bihC0IP}, we obtain
\[
    \|\mathbf{Q}_h\|_h^2
    +
    \sigma J_h(u_h,u_h)
    \le
    \|\mathbf{W}_h\|_h\|\mathbf{Q}_h\|_h .
\]
In particular, since \(J_h\ge0\), we have
\begin{equation}\label{eq:boundQh}
        (1 + \sigma s)\|\mathbf{Q}_h\|_h\le \|\mathbf{W}_h\|_h,\quad s :=
    \frac{J_h(u_h,u_h)}{\|\mathbf{Q}_h\|_h^2}.
\end{equation}
Therefore, we proved $\|R_h\|_{h\to h}\leq 1$. If \(\mathbf{Q}_h=0\), then the desired estimate is trivial. Hence, we assume \(\mathbf{Q}_h\neq0\). We decompose \(\mathbf{Q}_h\) into its tangent and normal components with respect to
\(T_{\mathcal{B}_h}(\widehat{\mathbf{P}}_h)\):
\[
    \mathbf{Z}_h := \Pi_{T_{\mathcal{B}_h}(\widehat{\mathbf{P}}_h)}^{(h)}\mathbf{Q}_h,
    \qquad
    \mathbf{N}_h := \mathbf{Q}_h-\mathbf{Z}_h .
\]
Since \(\Pi_{T_{\mathcal{B}_h}(\widehat{\mathbf{P}}_h)}^{(h)}\) is the \(\|\cdot\|_h\)-orthogonal projection, we have \(\|\mathbf{Q}_h\|_h^2=\|\mathbf{Z}_h\|_h^2+\|\mathbf{N}_h\|_h^2\). In particular,  
\begin{equation}\label{eq:boundzh}
    \|\mathbf{Z}_h\|_h
    =
    \sqrt{1-r^2}\,\|\mathbf{Q}_h\|_h ,\quad     r^2 :=
    \frac{\|\mathbf{N}_h\|_h^2}{\|\mathbf{Q}_h\|_h^2},
\end{equation}
with, by construction, \(0\le r^2\le1\). Using \eqref{eq:boundQh}, we obtain 
\[
    \|\mathbf{Z}_h\|_h
    \le
    \frac{\sqrt{1-r^2}}{1+\sigma s}
    \|\mathbf{W}_h\|_h .
\]
To conclude the proof, we need to verify that $\frac{\sqrt{1-r^2}}{1+\sigma s}\leq q_0<1$. We now relate the scalar quantity \(\mathbf{C}_h:\mathbf{Q}_h\) to the normal component \(\mathbf{N}_h\). By definition of the tangent space \(T_{\mathcal{B}_h}(\widehat{\mathbf{P}}_h)\), \(\mathbf{C}_h:\mathbf{Z}_h=0\) at every local node. Therefore, we get
\[
    \|\mathbf{C}_h:\mathbf{Q}_h\|_h^2
    =
    \|\mathbf{C}_h:\mathbf{N}_h\|_h^2
    \le
    2\nu_1^2\|\mathbf{N}_h\|_h^2 .
\]
Using \eqref{eq:MTstab} and \Cref{lemma:TMimpliesCordes}, we
obtain
\[
    \|\mathbf{Q}_h\|_h^2
    \le
    2C_D\nu_1^2\|\mathbf{N}_h\|_h^2
    +
    C_DJ_h(u_h,u_h).
\]
Dividing by \(\|\mathbf{Q}_h\|_h^2\) yields
\begin{equation}\label{eq:rs}
      1
    \le
    C_1 r^2 + C_2 s ,  
\end{equation}
where $C_1:=2C_D\nu_1^2$ and $C_2:=C_D$. We may assume \(C_1\ge1\) and
\(C_2\ge1\).

We distinguish three cases. First, we suppose $s>0$ and  \(r^2\ge \frac{1}{2C_1}\). Then
\[\sqrt{1-r^2}\le\sqrt{1-\frac{1}{2C_1}}.\]
Since \(1+\sigma s\ge1\), estimate \eqref{eq:boundzh} gives
\[
    \|\mathbf{Z}_h\|_h
    \le
    \sqrt{1-\frac{1}{2C_1}}
    \|\mathbf{W}_h\|_h .
\]
If $s>0$ and \(r^2<\frac{1}{2C_1}\), then \eqref{eq:rs} implies
\[
    1
    \le
    C_1r^2+C_2s
    <
    \frac12+C_2s \implies  s>\frac{1}{2C_2}.
\]
Therefore, we obtain
\[
    \|\mathbf{Z}_h\|_h
    \le
    \frac{1}{1+\sigma/(2C_2)}
    \|\mathbf{W}_h\|_h .
\]
Now, we suppose $s=0$, then $r^2\geq \frac{1}{C_1}$ and hence 
$$\|\mathbf{Z}_h\|_h \leq \sqrt{1 - \frac{1}{C_1}}\|\mathbf{W}_h\|_h.$$
Hence, we get
\[
    \|\mathbf{Z}_h\|_h
    \le
    q_0\|\mathbf{W}_h\|_h,\quad     q_0
    :=\max
    \left\{
        \sqrt{1-\frac{1}{2C_1}},
        \frac{1}{1+\sigma/(2C_2)}
    \right\}\in (0,1).
\]
Since
\[
    \mathbf{Z}_h
    =
    \Pi_{T_{\mathcal{B}_h}(\widehat{\mathbf{P}}_h)}^{(h)}\mathbf{Q}_h
    =
    \Pi_{T_{\mathcal{B}_h}(\widehat{\mathbf{P}}_h)}^{(h)}R_h\mathbf{W}_h ,
\]
we have proved
\[
    \left\|
        \Pi_{T_{\mathcal{B}_h}(\widehat{\mathbf{P}}_h)}^{(h)}R_h\mathbf{W}_h
    \right\|_h
    \le
    q_0\|\mathbf{W}_h\|_h
    \qquad
    \forall \mathbf{W}_h\in Z_h(\mathbb{S}^2) .
\]
Taking the supremum over all \(\|\mathbf{W}_h\|_h = 1\) gives the desired result.
\end{proof}

We now combine \Cref{thm:smooth_projection} and \Cref{lem:transverlality} to prove \Cref{thm:local_fixed_point}. The transversality estimate gives the contraction mechanism after projection onto \(\mathcal B_h\). The point requiring care is the compatibility between the two norms used in the argument. Indeed, the smoothness of the discrete projection \(\Pi_{\mathcal B_h}^{(h)}\), proved in \Cref{thm:smooth_projection}, holds in a neighbourhood defined with respect to the $\|\cdot\|_{h,\infty}$ norm. By contrast, the fixed-point argument is carried out in the \(\|\cdot\|_h\)-ball.

\begin{proof}[Proof of \Cref{thm:local_fixed_point}]
Let \(\mathbf{Q}_h\in \mathcal{U}_\kappa^h(\widehat{\mathbf{P}}_h)\). First, we prove that \(R_h(\mathbf{Q}_h)\in\mathcal{U}_\tau^{\infty}(\widehat{\mathbf{P}}_h)\) for \(h\) and $\kappa$ small enough. Indeed, using the inequality \eqref{eq:inverseest} and the fact that \(\|R_h\|_{h\to h}\leq 1\) by \Cref{lem:transverlality}, we obtain
\begin{align}\label{eq:estim_res}
\|R_h(\mathbf{Q}_h)-\widehat{\mathbf{P}}_h\|_{h,\infty}
&\leq
\|R_h(\mathbf{Q}_h-\widehat{\mathbf{P}}_h)\|_{h,\infty}
+
\|R_h(\widehat{\mathbf{P}}_h)-\widehat{\mathbf{P}}_h\|_{h,\infty}
\notag \\ 
&\leq
C_{\text{inv}}h^{-1}
\|R_h(\mathbf{Q}_h-\widehat{\mathbf{P}}_h)\|_h
+
C_{\text{inv}}h^{-1}\|R_h(\widehat{\mathbf{P}}_h)-\widehat{\mathbf{P}}_h\|_{h}
\notag \\ 
&\leq C_{\text{inv}}(\kappa + C_Rh^{\alpha}).
\end{align}
Hence, \(R_h(\mathbf{Q}_h)\in\mathcal{U}_\tau^{\infty}(\widehat{\mathbf{P}}_h)\) for $h$ and $\kappa$ small. We now bound the derivative of $T_h$ in $\mathcal{U}_\kappa^h(\widehat{\mathbf{P}}_h)$. Since
\[
DT_h(\mathbf{Q}_h)
=
D\Pi_{\mathcal{B}_h}^{(h)}(R_h(\mathbf{Q}_h))\circ R_h,
\]
by using \Cref{thm:smooth_projection}, \Cref{lem:transverlality} and \eqref{eq:estim_res}, we obtain
\begin{align*}
\|DT_h(\mathbf{Q}_h)\|_{h\to h}
&\leq
\|
\Pi^{(h)}_{T_{\mathcal{B}_h}(\widehat{\mathbf{P}}_h)}
\circ R_h
\|_{h\to h}
\\
&\quad
+
\|
(
D\Pi_{\mathcal{B}_h}^{(h)}(R_h(\mathbf{Q}_h))
-
D\Pi_{\mathcal{B}_h}^{(h)}(\widehat{\mathbf{P}}_h)
)
\circ R_h
\|_{h\to h}\\
&\leq
q_0
+
C_\Pi 
\|R_h(\mathbf{Q}_h)-\widehat{\mathbf{P}}_h\|_{h,\infty}
\\
&\leq
q_0
+
C_\Pi C_{\text{inv}}(\kappa + C_Rh^{\alpha}).
\end{align*}
Hence, for $\kappa$ and $h$ small, there exists $q\in (0,1)$ such that
\begin{equation}\label{eq:deri}
 \sup_{\mathbf{Q}_h\in \mathcal{U}_\kappa^h(\widehat{\mathbf{P}}_h)}
    \|DT_h(\mathbf{Q}_h)\|_{h\to h}
    \leq
    q<1 ,
\end{equation}
and 
\[
\|T_h(\mathbf{Q}_h)-T_h(\mathbf{Z}_h)\|_h
\leq
q
\|\mathbf{Q}_h-\mathbf{Z}_h\|_h
\qquad
\forall \mathbf{Q}_h,\mathbf{Z}_h\in \mathcal{U}_\kappa^h(\widehat{\mathbf{P}}_h).
\]
In order to apply the contraction mapping theorem to $T_h$, we have to prove that it maps \(\mathcal{U}_\kappa^h(\widehat{\mathbf{P}}_h)\) into itself. We define $\varepsilon_h := \|T_h(\widehat{\mathbf{P}}_h)-\widehat{\mathbf{P}}_h\|_h$ and we have:
\begin{align}\label{eq:resestimate}
\varepsilon_h 
    \nonumber    =
    \|\Pi_{\mathcal{B}_h}^{(h)}(R_h(\widehat{\mathbf{P}}_h))
    -
    \Pi_{\mathcal{B}_h}^{(h)}(\widehat{\mathbf{P}}_h)\|_h
    \nonumber
    &\leq
    C_\Pi
    \|R_h(\widehat{\mathbf{P}}_h)-\widehat{\mathbf{P}}_h\|_h
    \nonumber\\
    &=
    C_\Pi \rho_h
    \leq
    C_{\Pi} C_Rh^{1+\alpha}.
\end{align}
If $h$ is small enough, then \(\varepsilon_h \leq (1-q)\kappa h\).
Therefore, for every \(\mathbf{Q}_h\in \mathcal{U}_\kappa^h(\widehat{\mathbf{P}}_h)\),
\begin{align*}
\|T_h(\mathbf{Q}_h)-\widehat{\mathbf{P}}_h\|_h
&\leq
\|T_h(\mathbf{Q}_h)-T_h(\widehat{\mathbf{P}}_h)\|_h
+
\|T_h(\widehat{\mathbf{P}}_h)-\widehat{\mathbf{P}}_h\|_h
\\
&\leq
q\|\mathbf{Q}_h-\widehat{\mathbf{P}}_h\|_h
+
\varepsilon_h
\\
&\leq
q \kappa h+(1-q)\kappa h
=
\kappa h.
\end{align*}
Thus, \(T_h\) maps \(\mathcal{U}_\kappa^h(\widehat{\mathbf{P}}_h)\) into itself and \(T_h\) is a contraction on \(\mathcal{U}_\kappa^h(\widehat{\mathbf{P}}_h)\).  By Banach's fixed-point theorem, there exists a unique fixed point \(\mathbf{P}_h^\star\in \mathcal{U}_\kappa^h(\widehat{\mathbf{P}}_h)\). Furthermore,
\begin{align*}
\|\mathbf{P}_h^\star-\widehat{\mathbf{P}}_h\|_h
&\leq
\|T_h(\mathbf{P}_h^\star)-T_h(\widehat{\mathbf{P}}_h)\|_h
+
\|T_h(\widehat{\mathbf{P}}_h)-\widehat{\mathbf{P}}_h\|_h
\\
&\leq
q
\|\mathbf{P}_h^\star-\widehat{\mathbf{P}}_h\|_h
+
\varepsilon_h.
\end{align*}
Absorbing the first term gives \eqref{eq:errP}. The convergence of the iteration follows from the contraction property.
\end{proof}

\begin{remark}\label{rem:valuealpha}
In the proof of \Cref{thm:local_fixed_point}, the assumption \(\alpha>0\) is used to ensure that
\[
    h^{-1}\rho_h \leq C_R h^\alpha .
\]
The same argument remains valid in the borderline case \(\alpha=0\), provided that \(C_R\) is sufficiently small. As discussed in \Cref{sec:consistency}, the constant \(C_R\) may depend on \(|u|_{H^3(\Omega)}\). Therefore, under an additional smallness assumption on the corresponding norm of the exact solution, the convergence result can also be applied in the case \(\alpha=0\).
\end{remark}

The following corollary shows that, along the alternating minimization scheme, the distance to the fixed point can be estimated in terms of the computable residuals \( \mathbf P_h^n-R_h(\mathbf P_h^n) \) and \( \mathbf{P}_h^{n+1} -\mathbf{P}_h^n \).

\begin{corollary}\label{cor:res}
    Under the assumptions of \Cref{thm:local_fixed_point}, for $n\geq 1$, the following bounds hold: 
    \begin{align}
        \frac{1}{2}\|\mathbf{P}_h^{n+1} -\mathbf{P}_h^n\|_h\leq \|&\mathbf{P}^n_h - R_h(\mathbf{P}^n_h)\|_h\leq 2\|\mathbf{P}^n_h - \mathbf{P}^\star_h\|_h + C\rho_h \\
        \|\mathbf{P}^n_h - \mathbf{P}^\star_h\|_h\leq \frac{1}{1-q }\|&\mathbf{P}_h^{n+1} -\mathbf{P}_h^n\|_h\leq \frac{2}{1-q }\|R_h(\mathbf{P}_h^n)-\mathbf{P}_h^n\|_h. 
    \end{align}
\end{corollary}
\begin{proof}
This is a consequence of the contraction estimate, of the linearity of $R_h$ and of the following bound: 
\begin{equation*}
    \|\mathbf{P}_h^{n+1} - \mathbf{P}_h^{n}\|_h \leq 2 \|R_h(\mathbf{P}_h^n)-\mathbf{P}_h^n\|_h. 
    \qedhere
\end{equation*}
\end{proof}

The radius \(O(h)\) in \Cref{thm:local_fixed_point} is a consequence of the same mechanism that appears in other local finite element convergence analyses for fully nonlinear elliptic equations \cite{feng_neilan_mixed,brenner,quadratic_neilan,unified,awanou}, as discussed in \Cref{secc:related_works}. The nonlinear constraint, and in particular the preservation of the elliptic branch, is a pointwise condition on the discrete Hessian. In the present argument the smoothness of the nodal projection \(\Pi^{(h)}_{B_h}\) is available in a neighbourhood measured in \(\|\cdot\|_{h,\infty}\), whereas the stability of the reconstruction step is obtained in the discrete \(L^2\)-type norm \(\|\cdot\|_h\). Consequently, the inverse estimate \eqref{eq:inverseest} in two space dimensions forces the admissible \(\|\cdot\|_h\)-radius to be of order \(h\). 

Although the contraction neighbourhood \(\mathcal U_{\kappa_0}^h(\widehat{\mathbf P}_h)\) has radius of order \(h\), this could be compatible with a nested mesh-refinement strategy. Let \(h'=h/2\), and assume that the refinement is nested. Let
\[
\mathcal I_h^{h'}:Z_h(\mathbb S^2)\to Z_{h'}(\mathbb S^2)
\]
be the natural prolongation operator, obtained by extending each coarse-element polynomial to the refined mesh, i.e., by using the same polynomial on every child element. Then, using \Cref{thm:local_fixed_point} and the standard interpolation estimate \cite{brenner_fem}, we obtain
\[
\begin{aligned}
    \|\mathcal I_h^{h'}\mathbf P_h^\star-\widehat{\mathbf P}_{h'}\|_{h'} \lesssim 
    \|\mathbf P_h^\star-\widehat{\mathbf P}_h\|_h
    +
    \|\mathcal I_h^{h'}\widehat{\mathbf P}_h-\widehat{\mathbf P}_{h'}\|_{h'} 
        \lesssim 
    h^\beta,
\end{aligned}
\]
where $\beta:=\min\{m+1,\alpha+1\}$. More generally, if \(\mathbf P_h^n\) is the \(n\)-th iterate on the coarse
mesh and the initial datum belongs to the coarse contraction neighbourhood,
then
\[
\begin{aligned}
    \|\mathcal I_h^{h'}\mathbf P_h^n-\widehat{\mathbf P}_{h'}\|_{h'}
    \leq
    \|\mathcal I_h^{h'}(\mathbf P_h^n-\mathbf P_h^\star)\|_{h'}
    +
    \|\mathcal I_h^{h'}\mathbf P_h^\star-\widehat{\mathbf P}_{h'}\|_{h'} \lesssim     q^n h + h^\beta .
\end{aligned}
\]
Therefore, if \(\beta>1\), then \(h^\beta=o(h)\). Hence, by choosing \(n\) large enough and then \(h\) sufficiently small, we obtain \(\mathcal I_h^{h'}\mathbf P_h^n\in\mathcal U_{\kappa_0}^{h'}(\widehat{\mathbf P}_{h'})\). Thus the prolongated coarse iterate can be used as an initial datum for the fixed-point iteration on the refined mesh.

\begin{remark}[Extension to smooth uniformly elliptic operators]\label{rem:extension}
The proof of \Cref{thm:smooth_projection} is not specific to the determinant constraint. It only requires that, near the exact Hessian, the discrete constraint set be a smooth manifold admitting a smooth nearest-point projection. Let \(F:\mathbb S^2\to\mathbb R\) be of class \(C^3\), and consider
\[
\begin{cases}
F(D^2u)=0 & \text{in }\Omega,\\
u=g & \text{on }\partial\Omega.
\end{cases}
\]
If \(F\) is uniformly elliptic in a neighbourhood of \(D^2u\), then \(F'(D^2u)\) is uniformly spd, see for instance \cite{caffarelli_cabre}. At each local node \(x_{K,j}\), the constraint \(F(M)=0\) defines a smooth hypersurface of \(\mathbb S^2\) near \(\widehat{\mathbf P}_h(x_{K,j})\). Consequently, for \(h\) sufficiently small, the nodal projection onto the corresponding discrete constraint set is locally well defined and smooth. Assuming moreover an a priori bound on \(\|F'(\widehat{\mathbf P}_h)\|_{W^{1,\infty}}\), \Cref{lemma:TMimpliesCordes} remains valid, and the remainder of the fixed-point argument is unchanged. Thus, the local convergence result in \Cref{thm:local_fixed_point} extends to smooth uniformly elliptic fully nonlinear operators.
\end{remark}

\section{Some examples of methods verifying~\texorpdfstring{\eqref{eq:MTstab}}{(MTstab)} and~\texorpdfstring{\eqref{eq:boundhR_h}}{(boundhR_h)}}
\label{sec:discreteMT}
The convergence result in \Cref{thm:local_fixed_point} was stated in terms of two abstract assumptions on the discrete reconstruction operator \(R_h\): the Miranda-Talenti-type stability estimate \eqref{eq:MTstab} and the consistency condition \eqref{eq:boundhR_h}. The purpose of this section is to verify these assumptions for several standard finite element discretizations of the linear biharmonic step.

We recall that for simplicity, we restrict the presentation to homogeneous Dirichlet boundary conditions. We consider three choices of the pair \((H_h,J_h)\): a conforming \(C^1\) finite element method, a \(C^0\) interior-penalty method, and a discontinuous Galerkin method. In each case, the stability estimate \eqref{eq:MTstab} follows either directly from the continuous Miranda-Talenti inequality or from known discrete Miranda-Talenti estimates. Once this estimate holds, the consistency condition \eqref{eq:boundhR_h} follows, provided that the polynomial degree \(m\) is chosen sufficiently large (specifically, $m\geq 1$). Recall that \(m\) denotes the polynomial degree of the discrete Hessian, while \(u_h\) is approximated by polynomials of degree \(k=m+2\). Hence, the analysis requires \(k\ge3\). Nevertheless, quadratic $C^0$ and DG discretizations are also covered provided that $|u|_{H^3}$ is sufficiently small; see \Cref{rem:valuealpha,rem:m0c0,rem:m0dg}. Similar minimum-degree assumptions appear, for example, in \cite{brenner,awanou}.

\subsection{Preliminary notation and results}\label{sec:preliminary_res}

Let \(\mathcal T_h\) be a shape-regular and quasi-uniform triangulation of a convex polygonal domain \(\Omega\), as introduced in \Cref{sec:not1}. We denote by
\(\mathcal F_h^i\), \(\mathcal F_h^b\), and \(\mathcal F_h\) the sets of interior, boundary, and all edges, respectively.

Let \(F\in\mathcal F_h^i\) be an interior edge shared by two elements \(K^+\) and \(K^-\), namely \(F=\partial K^+\cap\partial K^-\). We denote by \(\mathbf n_\pm\) the outward unit normal to \(K^\pm\), and for a piecewise defined function \(w\), we write \(w^\pm:=w|_{K^\pm}\). For a vector field \(\mathbf w\), we define the normal jump across \(F\in\mathcal{F}^i_h\) by
\begin{equation}
    \llbracket \mathbf w \rrbracket |_F
    :=
    \mathbf w^+\cdot \mathbf n_+
    +
    \mathbf w^-\cdot \mathbf n_- .
\end{equation}
For a scalar function \(w\), the jump is vector-valued and is defined by
\begin{align}
    \llbracket w \rrbracket |_F
    &:=
    w^+\mathbf n_+
    +
    w^-\mathbf n_-,
    && F\in\mathcal F_h^i, \\
    \llbracket w \rrbracket |_F
    &:=
    w\mathbf n,
    && F\in\mathcal F_h^b,
\end{align}
where, on a boundary edge, \(\mathbf n\) denotes the outward unit normal to \(\Omega\). We shall use the standard trace inequality \cite{brenner_fem}: for every \(K\in\mathcal T_h\), every edge \(F\subset\partial K\), and every \(v\in H^1(K)\),
\begin{equation}\label{eq:trace_inequality}
    \|v\|_{L^2(F)}
    \lesssim
    h_K^{-1/2}\|v\|_{L^2(K)}
    +
    h_K^{1/2}|v|_{H^1(K)} .
\end{equation}

Finally, let \(I_K\) denote the local interpolation operator onto \(\mathbb P_\ell(K)\). We recall the following standard interpolation estimate for  \(1<p\leq\infty\) \cite[Theorem 4.4.4]{brenner_fem}:
\begin{equation}\label{eq:interp_estimates}
    |v-I_Kv|_{W^{i,p}(K)}
    \leq
    C h_K^{s-i}|v|_{W^{s,p}(K)} \quad \forall\,v\in W^{s,p}(K),
\end{equation}
with \(0\leq i\leq s\leq \ell+1\), \(s>2/p\) and $C>0$ independent of $h$. In what follows, \(I_h^X\) and \(I_h^Z\) denote the corresponding elementwise interpolation operators into \(X_h\) and \(Z_h\), respectively.

\subsection{Conforming and non-conforming finite element methods}\label{sec:consistency}
Before specifying the choices of \(X_h\), \(H_h\), and \(J_h\) for each method, we isolate a consistency argument that is common to all the conforming and non-conforming discretizations considered below.

Recall that \(\widehat{\mathbf P}_h := I_h^Z(D^2u)\in Z_h(\mathbb S^2)\), where \(u\) is the exact solution of the Monge-Amp\`ere equation satisfying \eqref{eq:elliptic}. Let \(w_h := I_h^X u \in X_h\) be a suitable interpolant of \(u\), and define
\[
    e_h := u_h(\widehat{\mathbf P}_h)-w_h \in X_h .
\]
Using \eqref{eq:bihC0IP}, we obtain
\[
    \|H_he_h\|_h
    +
    \sigma^{1/2}J_h(e_h,e_h)^{1/2}
    \lesssim
        \|\widehat{\mathbf P}_h-H_hw_h\|_h
        +
        \sigma^{1/2}J_h(w_h,w_h)^{1/2}
    .
\]
Therefore,
\begin{align}
    \rho_h
    &:=
    \|R_h(\widehat{\mathbf P}_h)-\widehat{\mathbf P}_h\|_h
      =
    \|H_hu_h(\widehat{\mathbf P}_h)-\widehat{\mathbf P}_h\|_h
    \notag\\
    &\le
    \|H_he_h\|_h
    +
    \|H_hw_h-\widehat{\mathbf P}_h\|_h
    \notag\\
    &\lesssim
        \|H_hw_h-\widehat{\mathbf P}_h\|_h
        +
        \sigma^{1/2}J_h(w_h,w_h)^{1/2}.
    \label{eq:rho_abstract_consistency}
\end{align}

It remains to estimate the two terms on the right-hand side. Since both \(H_hw_h\) and \(\widehat{\mathbf P}_h\) belong to \(Z_h(\mathbb S^2)\), the equivalence between the nodal norm and the broken \(L^2\)-norm and triangle inequality yield
\[
    \|H_hw_h-\widehat{\mathbf P}_h\|_h
    \lesssim
    \|H_hw_h-\widehat{\mathbf P}_h\|_{L^2(\mathcal T_h)} \lesssim    
    \|H_hw_h-D^2u\|_{L^2(\mathcal T_h)}
    +
    \|D^2u-\widehat{\mathbf P}_h\|_{L^2(\Omega)} .
\]
Moreover, by the interpolation estimate \eqref{eq:interp_estimates},
\[
    \|D^2u-\widehat{\mathbf P}_h\|_{L^2(\Omega)}
    =
    \|D^2u-I_h^Z(D^2u)\|_{L^2(\Omega)}
    \le
    C h^{m+1}|u|_{H^{m+3}(\Omega)} .
\]
Thus, after choosing a discretization for which the Miranda-Talenti stability estimate \eqref{eq:MTstab} holds, in order to prove \eqref{eq:boundhR_h}, it remains to control
\[
    \|H_hw_h-D^2u\|_{L^2(\mathcal T_h)}
    \qquad\text{and}\qquad
    J_h(w_h,w_h)^{1/2}.
\]
We now consider specific choices of \(X_h\), \(H_h\), and \(J_h\) for which \eqref{eq:MTstab} is valid, and we estimate the two quantities above in each case, taking into account the corresponding treatment of the boundary condition.

\subsubsection{$C^1$ finite elements}\label{sssec:c1}
We first consider the conforming \(C^1\) case as a benchmark example. It is the simplest setting analytically, because no stabilization is needed and the continuous Miranda-Talenti inequality can be applied directly. From the computational point of view, however, this case is less attractive, since \(C^1\) finite elements are typically more complicated to implement than
\(C^0\) or discontinuous finite elements.

We impose the homogeneous Dirichlet boundary condition strongly and choose a conforming \(C^1\) finite element space \(X_h \subset H^2(\Omega)\cap H^1_0(\Omega)\) of polynomial degree \(k=m+2\), namely: 
\[
X_h
:=
\bigl\{v\in H^2\cap H^1_0(\Omega): v|_K\in \mathbb{P}_k(K)
\ \text{for all } K\in\mathcal T_h\bigr\}.
\]
We define $H_hv_h:=D^2v_h$ and $J_h\equiv0$. With this choice, \eqref{eq:MTstab} reduces exactly to the continuous Miranda-Talenti inequality:
\[
    \|D^2v_h\|_{L^2(\Omega)}
    \le
    \|\Delta v_h\|_{L^2(\Omega)}
    \qquad
    \forall v_h\in X_h.
\]
This inequality holds because \(X_h\subset H^2(\Omega)\cap H^1_0(\Omega)\) and
\(\Omega\) is convex.

It remains to verify the consistency condition \eqref{eq:boundhR_h}. Let \(I_h^Xu\in X_h\) be the corresponding \(C^1\) interpolant of the exact solution. Since \(J_h\equiv0\), the abstract estimate \eqref{eq:rho_abstract_consistency} only requires us to control the interpolation error of the Hessian. By
\eqref{eq:interp_estimates},
\[
    \|D^2I_h^Xu-D^2u\|_{L^2(\Omega)}
    \le
    C h^{m+1}|u|_{H^{m+3}(\Omega)}.
\]
Therefore,
\[
    \rho_h
    \le
    C h^{m+1}|u|_{H^{m+3}(\Omega)}.
\]
Hence \eqref{eq:boundhR_h} holds with \(\alpha=m\). For instance, the Argyris element has polynomial degree \(k=5\), corresponding to \(m=3\). In this case, both \eqref{eq:MTstab} and \eqref{eq:boundhR_h} are satisfied, and therefore \Cref{thm:local_fixed_point} applies and the following corollary holds. 

\begin{corollary}\label{cor:conv_c1}
    Let $X_h$ be the space of Argyris elements of degree $k=5$. Then, if $H_hv_h := D^2v_h$, $v_h\in X_h$ and $J_h\equiv 0$, \eqref{eq:bihC0IP} is well posed. Moreover, define $u_h^\star:= u_h(\mathbf{P}_h^\star)$ as the solution to \eqref{eq:bihC0IP}. Then, $u_h^\star$ is convex and 
    $$\|H_hu_h^\star - D^2u\|_{L^2(\mathcal{T}_h)}\lesssim\, h^{1+m}.$$
\end{corollary}
\begin{proof}
Since \(X_h\subset H^2(\Omega)\cap H^1_0(\Omega)\), \(H_hv_h=D^2v_h\), and \(J_h\equiv0\), the bilinear form \(a_h\) is coercive on \(X_h\). Hence \eqref{eq:bihC0IP} is well posed. Since $u_h^\star\in C^1$, its convexity is a consequence of \Cref{thm:smooth_projection} and \Cref{thm:local_fixed_point}. Also, \Cref{thm:local_fixed_point} yields
\begin{align*}
    \|H_hu_h^\star-D^2u\|_{L^2(\mathcal T_h)}
&\lesssim
\|R_h(\mathbf{P}_h^\star)-R_h(\widehat{\mathbf{P}}_h)\|_h
+\|R_h(\widehat{\mathbf{P}}_h)-\widehat{\mathbf{P}}_h\|_h
+\|\widehat{\mathbf{P}}_h-D^2u\|_{L^2(\Omega)}  \\
&\lesssim
\|\mathbf{P}_h^\star-\widehat{\mathbf{P}}_h\|_h+\rho_h+h^{m+1}
\lesssim h^{m+1}. \qedhere
\end{align*}
\end{proof}

\subsubsection{$C^0$ finite elements}\label{sssec:c0}
We now consider \(C^0\) Lagrange finite elements. We impose the homogeneous Dirichlet boundary condition strongly and choose
\(X_h\subset H^1_0(\Omega)\) to be the \(C^0\) finite element space of polynomial degree \(k=m+2\), namely: 
\[
X_h
:=
\bigl\{v\in H^1_0(\Omega): v|_K\in \mathbb{P}_k(K)
\ \text{for all } K\in\mathcal T_h\bigr\}.
\]
Since functions in \(X_h\) are not globally \(H^2\), we define the discrete Hessian elementwise:
\begin{equation}\label{eq:element_wise_hessian}
    (H_hv_h)|_K := D^2(v_h|_K),
    \qquad K\in\mathcal T_h .
\end{equation}

The continuous Miranda-Talenti inequality cannot be applied directly, because \(X_h\not\subset H^2(\Omega)\). We therefore add the jump stabilization
\begin{equation}\label{eq:stabilization}
    J_h(v_h,w_h)
    :=
    \sum_{F\in\mathcal F_h^i}
    h_F^{-1}
    \int_F
    \llbracket \nabla v_h\rrbracket
    \llbracket \nabla w_h\rrbracket
    \,ds .
\end{equation}
For convex polygonal domains, the discrete Miranda-Talenti estimates of \cite[Theorem 4.8]{neilansalgado} and \cite[Theorem 1]{neilan-wu} imply that \eqref{eq:MTstab} holds.

It remains to verify the consistency condition \eqref{eq:boundhR_h}. Let \(I_h^X\) be the standard \(C^0\) Lagrange interpolation operator of degree \(k=m+2\), and set \(w_h := I_h^Xu \). By the interpolation estimate \eqref{eq:interp_estimates},
\[
    \|H_hw_h-D^2u\|_{L^2(\mathcal T_h)}
    \le
    C h^{m+1}|u|_{H^{m+3}(\Omega)} .
\]
We now estimate the stabilization term. Set \(\eta_h := I_h^Xu-u \). Since \(u\) is smooth, its gradient has no jump across interior faces, namely \(\llbracket \nabla u\rrbracket=0\). Hence, for every interior face \(F=K^+\cap K^-\), \(\llbracket \nabla I_h^Xu\rrbracket=\llbracket \nabla \eta_h\rrbracket\). Therefore,
\[
    \|\llbracket \nabla I_h^Xu\rrbracket\|_{L^2(F)}^2
    \le
    C\left(
        \|\nabla \eta_h\|_{L^2(F\cap\partial K^+)}^2
        +
        \|\nabla \eta_h\|_{L^2(F\cap\partial K^-)}^2
    \right).
\]
Applying the trace inequality \eqref{eq:trace_inequality} on each neighbouring element and the interpolation estimates \eqref{eq:interp_estimates}, yields
\begin{equation}\label{eq:jumpss}
    h_F^{-1}
    \|\llbracket \nabla I_h^Xu\rrbracket\|_{L^2(F)}^2
    \le
    C h^{2m+2}
    \sum_{K\in\{K^+,K^-\}}
    |u|_{H^{m+3}(K)}^2 .
\end{equation}
Summing over all interior faces gives
\[
    J_h(w_h,w_h)^{1/2}
    \le
    C h^{m+1}|u|_{H^{m+3}(\Omega)} .
\]
Combining this estimate with \eqref{eq:rho_abstract_consistency}, we obtain \(\rho_h\le C h^{m+1}|u|_{H^{m+3}(\Omega)}\). Thus \eqref{eq:boundhR_h} holds with \(\alpha=m\). In particular, if \(m\geq 1\), then the assumptions of \Cref{thm:local_fixed_point} are satisfied, and the map \(T_h\) is a contraction in the local neighbourhood of \(\widehat{\mathbf P}_h\) specified there. Moreover, the following corollary holds. 
\begin{corollary}\label{cor:conv_c0}
    Let $X_h\subset H^1_0(\Omega)$ be the space of $C^0$ finite elements of degree $k\geq 3$ . Then, if $H_hv_h$ and $J_h$ are respectively defined by \eqref{eq:element_wise_hessian} and \eqref{eq:stabilization}, \eqref{eq:bihC0IP} is well posed. Moreover, define $u_h^\star:= u_h(\mathbf{P}_h^\star)$ as the solution to \eqref{eq:bihC0IP}. Then, $u_h^\star$ is elementwise convex and
    $$\|H_hu_h^\star - D^2u\|_{L^2(\mathcal{T}_h)} + \sigma^{1/2}J_h(u_h^\star,u_h^\star)^{1/2}\lesssim\, h^{1+m}.$$
\end{corollary}
\begin{proof}
The well-posedness of \eqref{eq:bihC0IP} follows from the coercivity of \(a_h\): if \(a_h(v_h,v_h)=0\), then \(v_h\) is elementwise affine, its
gradient has no interior jump, and \(v_h=0\) on \(\partial\Omega\), hence \(v_h=0\). Elementwise convexity is a consequence of \Cref{thm:smooth_projection} and \Cref{thm:local_fixed_point}. Now, let \(w_h:=I_h^X u\) and set \(\epsilon_h:=u_h^\star-w_h\). Subtracting the equation satisfied by \(w_h\) from \eqref{eq:bihC0IP} with \(\mathbf{P}_h=\mathbf{P}_h^\star\), and testing with \(\epsilon_h\), gives 
\[ \|H_h\epsilon_h\|_h+\sigma^{1/2}J_h(\epsilon_h,\epsilon_h)^{1/2} \lesssim \|\mathbf{P}_h^\star-H_hw_h\|_h + \sigma^{1/2}J_h(w_h,w_h)^{1/2}.\]
Using the interpolation estimates and the bound on \(\mathbf{P}_h^\star-\widehat{\mathbf{P}}_h\), we obtain
\[
\|\mathbf{P}_h^\star-H_hw_h\|_h
\le
\|\mathbf{P}_h^\star-\widehat{\mathbf{P}}_h\|_h
+\|\widehat{\mathbf{P}}_h-D^2u\|_{L^2(\Omega)}
+\|D^2u-H_hw_h\|_{L^2(\mathcal T_h)}
\lesssim h^{m+1}.
\]
Furthermore, as shown above, \( J_h(w_h,w_h)^{1/2}\lesssim h^{m+1}\). Therefore
\[
\|H_h\epsilon_h\|_h+\sigma^{1/2}J_h(\epsilon_h,\epsilon_h)^{1/2}
\lesssim h^{m+1}.
\]
Finally,
\[
\|H_hu_h^\star-D^2u\|_{L^2(\mathcal T_h)}
\le
\|H_h\epsilon_h\|_{L^2(\mathcal T_h)}
+
\|H_hw_h-D^2u\|_{L^2(\mathcal T_h)}
\lesssim h^{m+1},
\]
and, by the triangle inequality for the seminorm induced by \(J_h\),
\[
J_h(u_h^\star,u_h^\star)^{1/2}
\le
J_h(\epsilon_h,\epsilon_h)^{1/2}+J_h(w_h,w_h)^{1/2}
\lesssim h^{m+1}.  \qedhere
\]
\end{proof}
\begin{remark}\label{rem:m0c0}
    If $m=0$ and $|u|_{H^3(\Omega)}$ is sufficiently small, \Cref{thm:local_fixed_point} holds, as pointed out in \Cref{rem:valuealpha}.
\end{remark}
\subsubsection{DG finite elements}\label{sssec:dg}
We finally consider a discontinuous Galerkin discretization. Specifically, let $X_h\subset L^2(\Omega)$ be the space of discontinuous finite elements of degree $k= m+2$, namely: 
\[
X_h
:=
\bigl\{v\in L^2(\Omega): v|_K\in \mathbb{P}_k(K)
\ \text{for all } K\in\mathcal T_h\bigr\}.
\]
In this case, no boundary condition is built into the discrete space. The homogeneous Dirichlet condition is instead imposed weakly through the boundary terms in the stabilization. For \(v_h\in X_h\), we define the discrete Hessian elementwise, as in \eqref{eq:element_wise_hessian}. Since \(X_h\not\subset H^2(\Omega)\), we add the stabilization
\begin{equation}\label{eq:stab_dg}
    J_h(v_h,w_h)
    :=
    \sum_{F\in\mathcal F_h^i}
    h_F^{-1}
    \int_F
    \llbracket \nabla v_h\rrbracket
    \llbracket \nabla w_h\rrbracket
    \,ds
    +
    \sum_{F\in\mathcal F_h}
    h_F^{-3}
    \int_F
    \llbracket v_h\rrbracket\cdot
    \llbracket w_h\rrbracket
    \,ds .
\end{equation}
Here the second sum is taken over all faces, including boundary faces; on boundary faces, the jump convention is precisely the one used to impose the homogeneous Dirichlet condition weakly. With this choice, the discrete Miranda-Talenti estimate of \cite[Theorem 6.1, Corollary 6.4]{kaweckismears} implies that \eqref{eq:MTstab} holds. Although their estimate involves the full gradient jump, its tangential component is controlled by the value-jump term through a standard inverse estimate; alternatively, one may define \(J_h\) exactly as in \cite{kaweckismears}, without affecting the subsequent analysis.

It remains to verify the consistency condition \eqref{eq:boundhR_h}. With respect to the $C^0$ term, we are left to estimate only the term of $J_h$ involving the jumps of the function. Let \(I_h^X\) be the elementwise interpolation operator of degree \(m+2\), and set \(w_h := I_h^Xu\) and \(\eta_h := w_h-u\). Similarly, we have \(\llbracket w_h\rrbracket=\llbracket \eta_h\rrbracket\) on $\mathcal{F}_h$. Applying the trace inequality \eqref{eq:trace_inequality} on each element adjacent to \(F\), and then using the interpolation estimates for \(\eta_h\), gives
\[
    \sum_{F\in\mathcal F_h}
    h_F^{-3}
    \|\llbracket w_h\rrbracket\|_{L^2(F)}^2
    \le
    C h^{2m+2}|u|_{H^{m+3}(\Omega)}^2 .
\]
Therefore, we obtain
\(\rho_h\le C h^{m+1}|u|_{H^{m+3}(\Omega)}\) and \eqref{eq:boundhR_h} holds with \(\alpha=m\geq 1\), as for the previous cases.

\begin{corollary}\label{cor:convdg}
    Let $X_h\subset L^2(\Omega)$ be the space of discontinuous finite elements of degree $k\geq 3$ . Then, if $H_hv_h$ and $J_h$ are respectively defined by \eqref{eq:element_wise_hessian} and \eqref{eq:stab_dg}, \eqref{eq:bihC0IP} is well posed. Moreover, define $u_h^\star:= u_h(\mathbf{P}_h^\star)$ as the solution to \eqref{eq:bihC0IP}. Then, $u_h^\star$ is elementwise convex and 
    $$\|H_hu_h^\star - D^2u\|_{L^2(\mathcal{T}_h)} + \sigma^{1/2}J_h(u_h^\star - u,u_h^\star - u)^{1/2}\lesssim\, h^{1+m}.$$
\end{corollary}
\begin{proof}
    Similar to the proof of \Cref{cor:conv_c0}.
\end{proof}
\begin{remark}\label{rem:m0dg}
    As for the $C^0$ setting, if $m=0$ and $|u|_{H^3(\Omega)}$ is sufficiently small, \Cref{thm:local_fixed_point} holds, as pointed out in \Cref{rem:valuealpha}.
\end{remark}

\subsubsection{Other constructions}

The framework above does not require \(H_h\) to be the elementwise Hessian. Other reconstruction operators can also be used, provided they satisfy the two abstract assumptions \eqref{eq:MTstab} and \eqref{eq:boundhR_h}. For instance, in \cite{zhang_Rec}, the authors introduce a recovery operator satisfying a Miranda-Talenti type estimate. To fit such a construction into the present framework, one would still need to verify the consistency condition \eqref{eq:boundhR_h}. Another related approach is \cite{gallistl_MT}, where a Miranda-Talenti inequality is proved for \(C^0\) finite elements without the stabilization term.

\section{Numerical experiments}\label{sec:numres}
\subsection{Setup}
We include some numerical experiments to illustrate the theoretical results, additional analysis is included in Appendix \ref{app:numres}. The linear problem \eqref{eq:bihC0IP} is solved using \texttt{fenicsx} \cite{BarattaEtal2023}, while the nonlinear projection \eqref{eq:firstmin_disc} is computed pointwise with the algorithm \texttt{Qmin} from \cite{glowinski}. The local \texttt{Qmin} problems are solved to a tolerance much smaller than the discretization error, making the projection error negligible at the reported scales. We consider both the \(C^0\) interior-penalty and DG methods with \(m=0,1,2\), and choose the scalar finite element degree as \(k=m+2\) (i.e., \(k=2,3,4\), respectively). The convergence theory applies directly to \(m=1,2\), while the case \(m=0\) requires the additional smallness condition from \Cref{rem:valuealpha}; we include it to illustrate the behaviour of the lowest-order discretization.

We compare two choices of local points \(x_{K,j}\) used to define the nodal norm \(\|\cdot\|_h\). The first one is the natural nodal set associated with the degrees of freedom: the barycenter for \(m=0\), the vertices for \(m=1\), and the vertices together with the edge midpoints for \(m=2\). The second one uses quadrature points exact for polynomials of degree \(2m\). Since products of functions in \(Z_h\) have degree at most \(2m\) on each element, this second choice makes the nodal scalar product coincide with the broken \(L^2\)-inner product on \(Z_h\). For \(m=1\), the \(2m\)-exact choice is the standard three-point symmetric Gauss rule, while for \(m=2\), we use the six-point Dunavant rule of degree \(4\) \cite{dunavant}. The different choices are illustrated in \Cref{fig:nodal_choices}.

Unless otherwise stated, we set \(\sigma=10\). Other positive values of \(\sigma\) give comparable $h$ convergence rates, although the iterative convergence rate depends on \(\sigma\); see Appendix \Cref{app:numres}. The nonhomogeneous Dirichlet condition is imposed strongly for the \(C^0\) interior-penalty discretization and weakly for the DG discretization. In the latter case, the boundary penalty term involves \(u_h-g\), leaving the bilinear form unchanged and adding the contribution
\[
\sigma\sum_{F\in\mathcal F_h^b}
h_F^{-3}\int_F g\,v_h\,ds
\]
to the right-hand side of the linear reconstruction problem.

To initialize the algorithm, we follow \cite{caboussat} and compute \(u_h^0\in X_h\) as the finite element approximation of
\begin{equation}\label{eq:init}
    \Delta u^0 = 2\sqrt f
    \quad\text{in }\Omega,
    \qquad
    u^0=g
    \quad\text{on }\partial\Omega .
\end{equation}
This provides a good initial guess when the eigenvalues of \(D^2u\) are close. We then set \(\mathbf P_h^0:=\Pi^{(h)}_{\mathcal B_h}(H_hu_h^0)\). We also use a nested mesh-refinement continuation strategy, initializing each refinement level with the solution from the previous one, as discussed in \Cref{sec:conv}. The mesh-refinement is uniform. Finally, motivated by \Cref{thm:local_fixed_point,cor:res}, we stop the iteration when
\[
\frac{\|\mathbf{P}^{n+1}_h-\mathbf{P}^{n}_h\|_h}
{\|\mathbf{P}^{n}_h\|_h}
\leq \delta h^{1+m},
\qquad
\delta=10^{-3}.
\]
Smaller values of \(\delta\) do not affect the observed convergence rates.

\begin{figure}[t]
\centering
\begin{tikzpicture}[
    scale=1.65,
    tri/.style={line width=0.45pt},
    nodept/.style={circle,fill=black,inner sep=1.35pt},
    quadpt/.style={circle,draw=black,fill=white,line width=0.45pt,inner sep=1.35pt},
    lab/.style={font=\scriptsize,align=center}
]

\begin{scope}[shift={(0,0)}]
\draw[tri] (0,0)--(1,0)--(0,1)--cycle;
\node[nodept] at (0.333333,0.333333) {};
\node[lab] at (0.5,-0.22) {\(m=0\)\\ barycenter};
\end{scope}

\begin{scope}[shift={(1.45,0)}]
\draw[tri] (0,0)--(1,0)--(0,1)--cycle;
\node[nodept] at (0,0) {};
\node[nodept] at (1,0) {};
\node[nodept] at (0,1) {};
\node[lab] at (0.5,-0.22) {\(m=1\)\\ \(\mathbb{P}_1\) nodes};
\end{scope}

\begin{scope}[shift={(2.90,0)}]
\draw[tri] (0,0)--(1,0)--(0,1)--cycle;
\node[quadpt] at (0.166667,0.166667) {};
\node[quadpt] at (0.666667,0.166667) {};
\node[quadpt] at (0.166667,0.666667) {};
\node[lab] at (0.5,-0.3) {\(m=1\)\\ 3-point \\ Gauss rule};
\end{scope}

\begin{scope}[shift={(4.35,0)}]
\draw[tri] (0,0)--(1,0)--(0,1)--cycle;
\node[nodept] at (0,0) {};
\node[nodept] at (1,0) {};
\node[nodept] at (0,1) {};
\node[nodept] at (0.5,0) {};
\node[nodept] at (0.5,0.5) {};
\node[nodept] at (0,0.5) {};
\node[lab] at (0.5,-0.22) {\(m=2\)\\ \(\mathbb{P}_2\) nodes};
\end{scope}

\begin{scope}[shift={(5.80,0)}]
\draw[tri] (0,0)--(1,0)--(0,1)--cycle;

\def\a{0.445948490915965}
\def\b{0.108103018168070}
\def\c{0.091576213509771}
\def\d{0.816847572980459}

\node[quadpt] at (\a,\b) {};
\node[quadpt] at (\b,\a) {};
\node[quadpt] at (\a,\a) {};
\node[quadpt] at (\c,\d) {};
\node[quadpt] at (\d,\c) {};
\node[quadpt] at (\c,\c) {};

\node[lab] at (0.5,-0.3) {\(m=2\)\\ 6-point \\ Dunavant rule};
\end{scope}

\node[nodept] at (1.0,-0.78) {};
\node[lab,anchor=west] at (1.12,-0.78) {\(m\)-exact quadrature points};

\node[quadpt] at (4.10,-0.78) {};
\node[lab,anchor=west] at (4.22,-0.78) {\(2m\)-exact quadrature points};

\end{tikzpicture}
\caption{Local points used to define the nodal norm on the reference triangle.}
\label{fig:nodal_choices}
\end{figure}
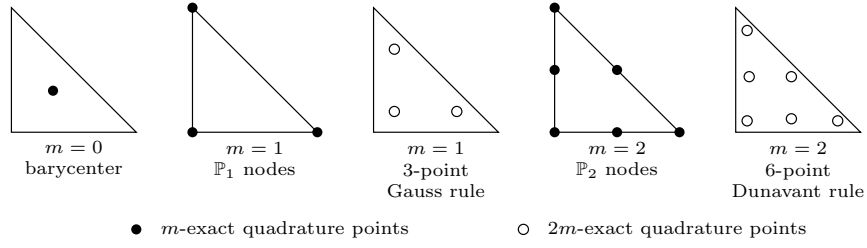

\begin{figure}[t]
    \centering
    \begin{subfigure}{0.99\linewidth}
    \centering
    \includegraphics[width = \linewidth]{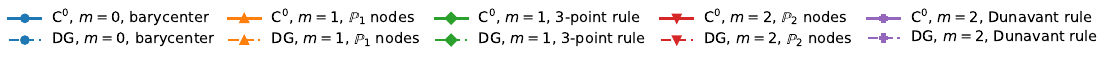}
\end{subfigure}
\vspace{0.1em}

\makebox[\linewidth]{\textbf{Test case 1}}

\vspace{0.3em}
\begin{subfigure}{0.24\linewidth}
    \includegraphics[width=0.99\linewidth]{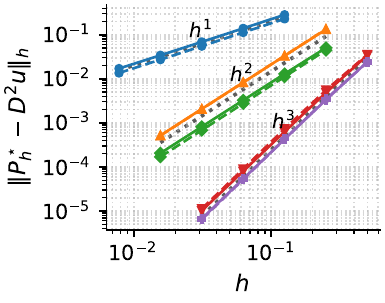}
\end{subfigure}
\begin{subfigure}{0.24\linewidth}
    \includegraphics[width=0.99\linewidth]{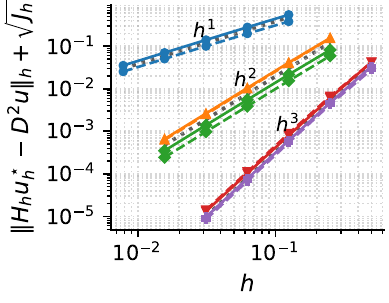}
\end{subfigure}
\begin{subfigure}{0.24\linewidth}
    \includegraphics[width=0.99\linewidth]{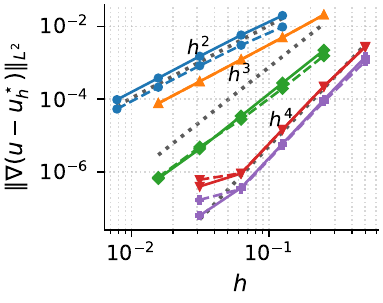}
\end{subfigure}
\begin{subfigure}{0.24\linewidth}
    \includegraphics[width=0.99\linewidth]{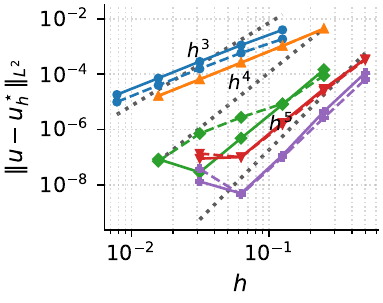}
\end{subfigure}
\vspace{0.01em}

\makebox[\linewidth]{\textbf{Test case 2}}

\vspace{0.3em}

\begin{subfigure}{0.24\linewidth}
    \includegraphics[width=0.99\linewidth]{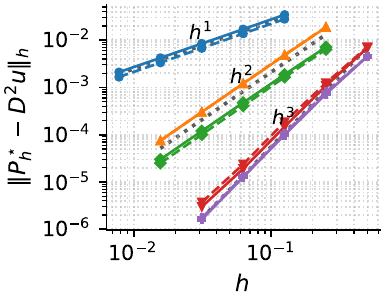}
\end{subfigure}
\begin{subfigure}{0.24\linewidth}
    \includegraphics[width=0.99\linewidth]{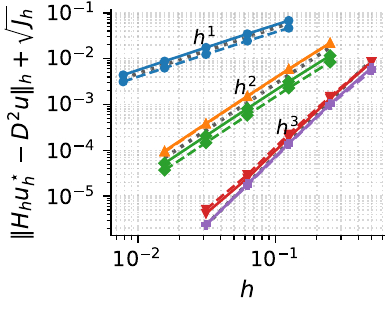}
\end{subfigure}
\begin{subfigure}{0.24\linewidth}
    \includegraphics[width=0.99\linewidth]{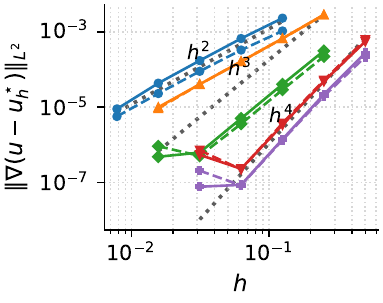}
\end{subfigure}
\begin{subfigure}{0.24\linewidth}
    \includegraphics[width=0.99\linewidth]{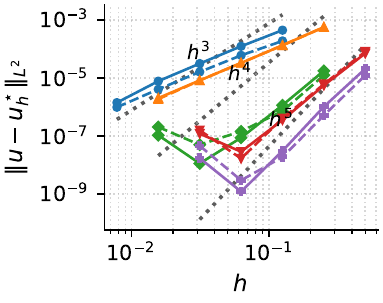}
\end{subfigure}
\caption{Errors vs. mesh size $h$ for test cases 1 (top row) and 2 (bottom row). From left to right: the error $\|\mathbf{P}_h^\star-D^2u\|_h$, the error $\|H_hu_h^\star-D^2u\|_h + J_h(u_h^\star-u,u_h^\star-u)^{1/2}$, the error $\|\nabla(u-u_h^\star)\|_{L^2}$ and the error $\|u-u_h^\star\|_{L^2}$.}
\label{fig:hconv}
\end{figure}

\begin{figure}[t]
    \centering
    \begin{subfigure}{0.99\linewidth}
    \centering
    \includegraphics[width = 0.55\linewidth]{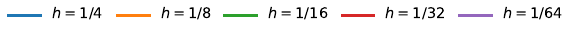}
\end{subfigure}
\vspace{0.1em}

\makebox[\linewidth]{\textbf{Initialization by continuation}}

\vspace{0.1em}
\begin{subfigure}{0.49\linewidth}
    \centering
    \includegraphics[width=0.48\linewidth]{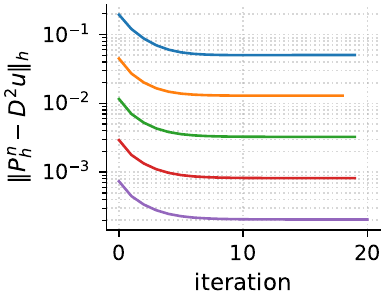}
    \includegraphics[width=0.48\linewidth]{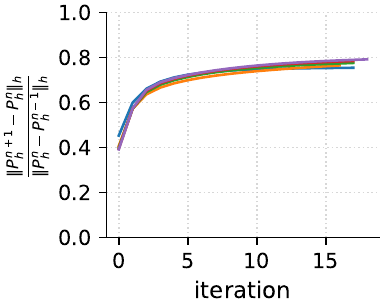}
    \caption*{\centering $C^0$, $m=1$, three-point Gauss rule}
\end{subfigure}
\begin{subfigure}{0.49\linewidth}
    \centering
    \includegraphics[width=0.48\linewidth]{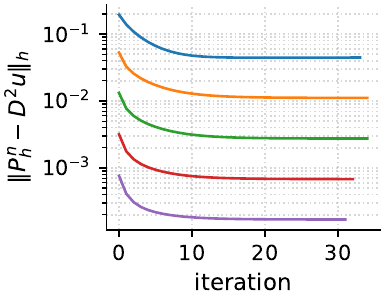}
    \includegraphics[width=0.48\linewidth]{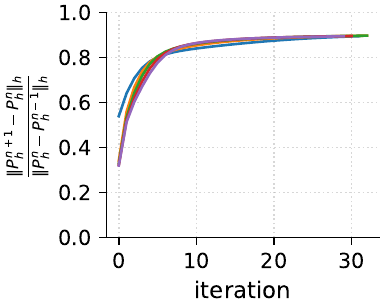}
    \caption*{\centering $DG$, $m=1$, three-point Gauss rule}
\end{subfigure}
\vspace{0.01em}
\makebox[\linewidth]{\textbf{Initialization by Poisson}}

\vspace{0.1em}

\begin{subfigure}{0.49\linewidth}
    \centering
    \includegraphics[width=0.48\linewidth]{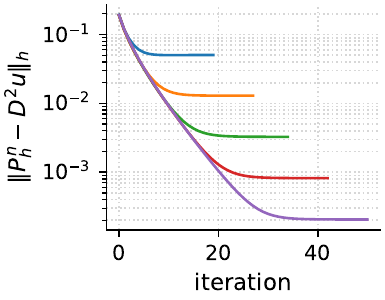}
    \includegraphics[width=0.48\linewidth]{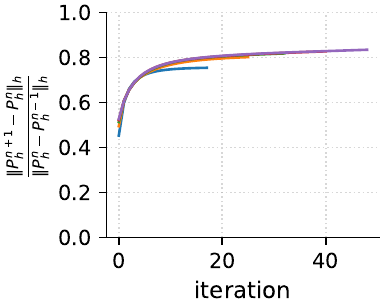}
    \caption*{\centering $C^0$, $m=1$, three-point Gauss rule}
\end{subfigure}
\begin{subfigure}{0.49\linewidth}
    \centering
    \includegraphics[width=0.48\linewidth]{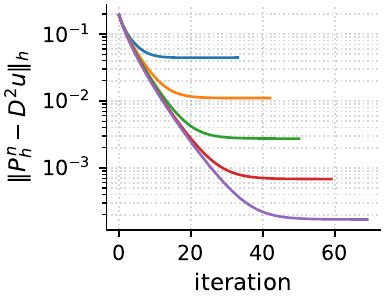}
    \includegraphics[width=0.48\linewidth]{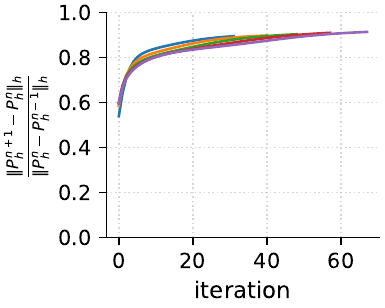}
    \caption*{\centering $DG$, $m=1$, three-point Gauss rule}
\end{subfigure}
\caption{Error and approximated contraction factor vs. iteration $n$ for two initialization strategies: continuation (top row) and Poisson (bottom row). From left to right: the error $\|\mathbf{P}_h^n-D^2u\|_h$ and the ratio  $\|\mathbf{P}_h^{n+1}-\mathbf{P}_h^n\|_h/\|\mathbf{P}_h^{n}-\mathbf{P}_h^{n-1}\|_h$ for $C^0$ finite elements with $m=1$ and the three-point Gauss rule and for DG finite elements with $m=1$ and the three-point Gauss rule.}
\label{fig:nconv_cont}
\end{figure}

\subsection{Test cases}
We consider two manufactured solutions on \(\Omega=[0,1]^2\):
\[
u_1(x,y)=\exp\!\left(\frac{x^2+y^2}{2}\right),
\qquad
u_2(x,y)=-\sqrt{4-x^2-y^2}.
\]
In both cases \(f=\det D^2u_i\) and \(g=u_i|_{\partial\Omega}\). \Cref{fig:hconv} reports the errors as functions of the mesh size $h$ for both test cases. For both the $C^0$ and DG methods, the errors $\|\mathbf{P}_h^\star-D^2u\|_h$ and $\|H_hu_h^\star-D^2u\|_h + J_h(u_h^\star - u,u_h^\star - u)^{1/2}$ converge with the expected rate $h^{m+1}$, including the case $m=0$. For \(m=1,2\), these observations agree with \Cref{thm:local_fixed_point} and \Cref{cor:conv_c0,cor:convdg}. For \(m=0\), the same rate is observed numerically, although the smallness condition required in \Cref{rem:m0c0,rem:m0dg} has not been checked for the present test cases.

The convergence of the $H^1$ and $L^2$-errors, \(\|\nabla(u-u_h^\star)\|_{L^2(\mathcal T_h)}\) and \(\|u-u_h^\star\|_{L^2(\mathcal T_h)}\) is more sensitive to the choice of the nodal points defining the discrete norm $\|\cdot\|_h$. For $m=1$, the natural choice of nodal points, corresponding to a quadrature rule that is exact on $\mathbb P_1$, does not yield the optimal rates for elliptic problems. In contrast, using quadrature points that are exact on $\mathbb P_2$ significantly improves the results. In the $C^0$ case, this recovers the optimal convergence rates $h^{m+2}$ and $h^{m+3}$ for the $H^1$- and $L^2$-errors, respectively. For the DG method, the optimal $H^1$-rate is also observed, whereas the $L^2$-error converges more slowly than expected. This behavior may be explained by the poor conditioning of the corresponding linear systems; see Appendix \ref{app:condition_numbers}. For $m=2$, both choices of nodal points recover the expected $H^1$-convergence rate $h^{m+2}=h^4$, although a slight deterioration is visible on the finest meshes, again consistent with the large condition numbers reported in Appendix \ref{app:condition_numbers}. The optimal $L^2$-rate $h^{m+3}=h^5$ is obtained only when the Dunavant quadrature rule is used, and deteriorates on the finest meshes, where the corresponding linear systems are also increasingly ill-conditioned. Finally, for the lowest-order case $m=0$, both the $H^1$ and $L^2$-errors converge at a rate close to $h^2$. While this is the optimal rate for the $H^1$-error, it is suboptimal for the $L^2$-error, whose expected rate would be $h^3$.

\Cref{fig:nconv_cont} reports the error $\|\mathbf{P}_h^n-D^2u\|_h$ and the approximated contraction factor $\|\mathbf{P}_h^{n+1}-\mathbf{P}_h^n\|_h/\|\mathbf{P}_h^{n}-\mathbf{P}_h^{n-1}\|_h$ as functions of the iteration of the splitting algorithm. It also compares the two initialization strategies: the continuation strategy and the initialization based on the Poisson problem \eqref{eq:init}. We present only the first test case with $m=1$ and the three-point Gauss rule, since the second test case and the other parameter choices exhibit similar behavior. For both initialization strategies, the error decreases monotonically with the iteration count. Moreover, after an initial transient, the ratios $\|\mathbf{P}_h^{n+1}-\mathbf{P}_h^n\|_h/\|\mathbf{P}_h^{n}-\mathbf{P}_h^{n-1}\|_h$ approach comparable values for the different mesh sizes, providing numerical evidence for a mesh-uniform linear convergence rate, consistently with \Cref{thm:local_fixed_point}. As expected, the continuation strategy requires significantly fewer iterations than the Poisson initialization. Nevertheless, the latter also converges to the same accuracy, albeit more slowly, suggesting that the practical basin of attraction may be larger than the sufficient neighbourhood provided by \Cref{thm:local_fixed_point}.

\section{Conclusions}
We developed a unified framework for the fully discrete convergence analysis of the least-squares splitting method for the Monge-Amp\`ere equation. Under suitable assumptions on the discretization of the linear variational subproblem, namely a discrete Miranda-Talenti inequality and a sufficiently high polynomial degree, we proved local convergence of the iterative method and optimal-order convergence of its fixed point to the exact solution in an \(H^2\)-type norm. We verified these assumptions for representative \(C^1\), \(C^0\), and DG finite element methods using existing discrete Miranda-Talenti estimates. The numerical experiments agree with the theoretical results. In particular, they confirm the predicted convergence rates and indicate that, in the tested cases, the practical basin of attraction is larger than the sufficient \(O(h)\)-neighbourhood established by the analysis.

Although the present analysis does not cover the original low-order discretizations of \cite{caboussat,peruso}, the framework applies to any reconstruction satisfying the abstract stability and consistency assumptions. It therefore also provides a route to the analysis of other existing or future low-order schemes, once suitable discrete Miranda-Talenti and consistency estimates are established.

Several other directions remain open for future research. These include Monge-Amp\`ere equations with gradient-dependent right-hand sides and second boundary conditions, in particular those arising in optimal transport and related problems. Another natural direction is the extension of the framework to other uniformly elliptic fully nonlinear operators.

\section*{Acknowledgments}
The author is grateful to Mohamed Ben Abdelouahab, Alexandre Caboussat and Marco Picasso for fruitful discussions. 

\appendix

\crefname{section}{appendix}{appendices}
\Crefname{section}{Appendix}{Appendices}
\section{Additional numerical results}\label{app:numres}

We report some additional numerical results to complement those in the main body of the manuscript. 
\subsection{Condition numbers}\label{app:condition_numbers}
The matrices arising from the assembly of \(a_h\) become increasingly ill-conditioned under mesh refinement, with observed condition numbers growing approximately as \(h^{-4}\), see \Cref{tab:condition_numbers}. The condition number is slightly better for $\mathbb{P}_{2m}$ exact quadrature nodes.

\subsection{Comparison of different values of $\sigma$}\label{app:sigma}
We investigate the influence of the parameter $\sigma\in\{1,2,5,20\}$ for the first test case. \Cref{fig:sigma_h} reports the $H^2$-type error as a function of the mesh size $h$, while \Cref{fig:sigma_iter} shows the convergence history together with the ratio $\|\mathbf{P}_h^{n+1}-\mathbf{P}_h^n\|_h/\|\mathbf{P}_h^{n}-\mathbf{P}_h^{n-1}\|_h$. The results indicate that, although $\sigma$ affects the convergence rate, consistently with \Cref{lem:transverlality}, the iteration converges for all values considered, and the convergence with respect to $h$ is essentially unchanged.

\begin{table}[h]
\centering
\caption{Condition numbers of the matrices associated with $a_h$. For each value of $m$, the columns $h_1,h_2,h_3$ denote the three finest mesh sizes used: $m=0$: $1/32$, $1/64$, $1/128$; $m=1$: $1/16$, $1/32$, $1/64$; $m=2$: $1/8$, $1/16$, $1/32$.}
\label{tab:condition_numbers}

\medskip
\textbf{$\sigma=1$}\\
\medskip

\begin{tabular}{c | c | c c c | c c c}
 & & \multicolumn{3}{c|}{$C^0$} & \multicolumn{3}{c}{DG} \\
\hline
$m$ & points & $h_1$ & $h_2$ & $h_3$ & $h_1$ & $h_2$ & $h_3$ \\
\hline
$0$ & barycenter & $10^{7}$ & $2\times 10^{8}$ & $3\times 10^{9}$ & $3\times 10^{7}$ & $4\times 10^{8}$ & $7\times 10^{9}$ \\
\hline
$1$ & $\mathbb{P}_1$ nodes & $2\times 10^{7}$ & $3\times 10^{8}$ & $4\times 10^{9}$ & $6\times 10^{7}$ & $9\times 10^{8}$ & $10^{10}$ \\
$1$ & Gauss & $6\times 10^{6}$ & $10^{8}$ & $2\times 10^{9}$ & $2\times 10^{7}$ & $4\times 10^{8}$ & $6\times 10^{9}$ \\
\hline
$2$ & $\mathbb{P}_2$ nodes & $10^{7}$ & $2\times 10^{8}$ & $3\times 10^{9}$ & $3\times 10^{7}$ & $5\times 10^{8}$ & $7\times 10^{9}$ \\
$2$ & Dunavant & $2\times 10^{6}$ & $4\times 10^{7}$ & $6\times 10^{8}$ & $9\times 10^{6}$ & $10^{8}$ & $2\times 10^{9}$ \\
\end{tabular}

\bigskip
\textbf{$\sigma=2$}\\
\medskip
\begin{tabular}{c | c | c c c | c c c}
 & & \multicolumn{3}{c|}{$C^0$} & \multicolumn{3}{c}{DG} \\
\hline
$m$ & points & $h_1$ & $h_2$ & $h_3$ & $h_1$ & $h_2$ & $h_3$ \\
\hline
$0$ & barycenter & $10^{7}$ & $2\times 10^{8}$ & $2\times 10^{9}$ & $2\times 10^{7}$ & $4\times 10^{8}$ & $6\times 10^{9}$ \\
\hline
$1$ & $\mathbb{P}_1$ nodes & $10^{7}$ & $2\times 10^{8}$ & $3\times 10^{9}$ & $4\times 10^{7}$ & $6\times 10^{8}$ & $10^{10}$ \\
$1$ & Gauss & $6\times 10^{6}$ & $9\times 10^{7}$ & $10^{9}$ & $2\times 10^{7}$ & $3\times 10^{8}$ & $4\times 10^{9}$ \\
\hline
$2$ & $\mathbb{P}_2$ nodes & $9\times 10^{6}$ & $10^{8}$ & $2\times 10^{9}$ & $2\times 10^{7}$ & $3\times 10^{8}$ & $5\times 10^{9}$ \\
$2$ & Dunavant & $2\times 10^{6}$ & $3\times 10^{7}$ & $5\times 10^{8}$ & $6\times 10^{6}$ & $9\times 10^{7}$ & $10^{9}$ \\
\end{tabular}

\bigskip
\textbf{$\sigma=5$}\\
\medskip

\begin{tabular}{c | c | c c c | c c c}
 & & \multicolumn{3}{c|}{$C^0$} & \multicolumn{3}{c}{DG} \\
\hline
$m$ & points & $h_1$ & $h_2$ & $h_3$ & $h_1$ & $h_2$ & $h_3$ \\
\hline
$0$ & barycenter & $10^{7}$ & $2\times 10^{8}$ & $3\times 10^{9}$ & $3\times 10^{7}$ & $4\times 10^{8}$ & $7\times 10^{9}$ \\
\hline
$1$ & $\mathbb{P}_1$ nodes & $10^{7}$ & $2\times 10^{8}$ & $3\times 10^{9}$ & $3\times 10^{7}$ & $4\times 10^{8}$ & $7\times 10^{9}$ \\
$1$ & Gauss & $6\times 10^{6}$ & $10^{8}$ & $2\times 10^{9}$ & $2\times 10^{7}$ & $2\times 10^{8}$ & $4\times 10^{9}$ \\
\hline
$2$ & $\mathbb{P}_2$ nodes & $8\times 10^{6}$ & $10^{8}$ & $2\times 10^{9}$ & $2\times 10^{7}$ & $2\times 10^{8}$ & $4\times 10^{9}$ \\
$2$ & Dunavant & $2\times 10^{6}$ & $3\times 10^{7}$ & $5\times 10^{8}$ & $5\times 10^{6}$ & $7\times 10^{7}$ & $10^{9}$ \\
\end{tabular}

\bigskip
\textbf{$\sigma=10$}\\
\medskip
\begin{tabular}{c | c | c c c | c c c}
 & & \multicolumn{3}{c|}{$C^0$} & \multicolumn{3}{c}{DG} \\
\hline
$m$ & points & $h_1$ & $h_2$ & $h_3$ & $h_1$ & $h_2$ & $h_3$ \\
\hline
$0$ & barycenter & $2\times 10^{7}$ & $3\times 10^{8}$ & $5\times 10^{9}$ & $4\times 10^{7}$ & $6\times 10^{8}$ & $10^{10}$ \\
\hline
$1$ & $\mathbb{P}_1$ nodes & $10^{7}$ & $2\times 10^{8}$ & $3\times 10^{9}$ & $3\times 10^{7}$ & $4\times 10^{8}$ & $7\times 10^{9}$ \\
$1$ & Gauss & $8\times 10^{6}$ & $10^{8}$ & $2\times 10^{9}$ & $2\times 10^{7}$ & $3\times 10^{8}$ & $5\times 10^{9}$ \\
\hline
$2$ & $\mathbb{P}_2$ nodes & $8\times 10^{6}$ & $10^{8}$ & $2\times 10^{9}$ & $10^{7}$ & $2\times 10^{8}$ & $3\times 10^{9}$ \\
$2$ & Dunavant & $2\times 10^{6}$ & $4\times 10^{7}$ & $6\times 10^{8}$ & $5\times 10^{6}$ & $8\times 10^{7}$ & $10^{9}$ \\
\end{tabular}

\bigskip
\textbf{$\sigma=20$}\\
\medskip

\begin{tabular}{c | c | c c c | c c c}
 & & \multicolumn{3}{c|}{$C^0$} & \multicolumn{3}{c}{DG} \\
\hline
$m$ & points & $h_1$ & $h_2$ & $h_3$ & $h_1$ & $h_2$ & $h_3$ \\
\hline
$0$ & barycenter & $3\times 10^{7}$ & $5\times 10^{8}$ & $8\times 10^{9}$ & $6\times 10^{7}$ & $9\times 10^{8}$ & $2\times 10^{10}$ \\
\hline
$1$ & $\mathbb{P}_1$ nodes & $2\times 10^{7}$ & $3\times 10^{8}$ & $5\times 10^{9}$ & $3\times 10^{7}$ & $5\times 10^{8}$ & $8\times 10^{9}$ \\
$1$ & Gauss & $10^{7}$ & $2\times 10^{8}$ & $3\times 10^{9}$ & $3\times 10^{7}$ & $4\times 10^{8}$ & $6\times 10^{9}$ \\
\hline
$2$ & $\mathbb{P}_2$ nodes & $9\times 10^{6}$ & $10^{8}$ & $2\times 10^{9}$ & $2\times 10^{7}$ & $2\times 10^{8}$ & $4\times 10^{9}$ \\
$2$ & Dunavant & $3\times 10^{6}$ & $5\times 10^{7}$ & $9\times 10^{8}$ & $6\times 10^{6}$ & $10^{8}$ & $2\times 10^{9}$ \\
\end{tabular}

\end{table}

\begin{figure}[h]
\centering

\begin{subfigure}{0.24\linewidth}
    \centering
    \small $\quad\quad\sigma=1$
\end{subfigure}
\begin{subfigure}{0.24\linewidth}
    \centering
    \small $\quad\quad\sigma=2$
\end{subfigure}
\begin{subfigure}{0.24\linewidth}
    \centering
    \small $\quad\quad\sigma=5$
\end{subfigure}
\begin{subfigure}{0.24\linewidth}
    \centering
    \small $\quad\quad\sigma=20$
\end{subfigure}

\vspace{0.3em}

\begin{subfigure}{0.24\linewidth}
    \includegraphics[width=0.99\linewidth]{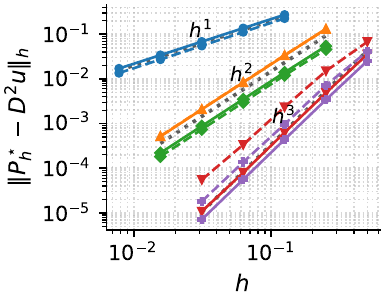}
\end{subfigure}
\begin{subfigure}{0.24\linewidth}
    \includegraphics[width=0.99\linewidth]{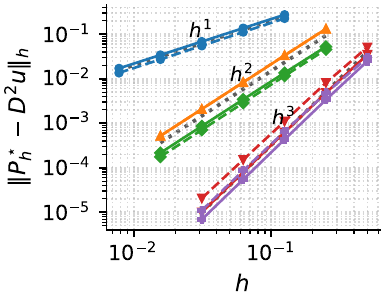}
\end{subfigure}
\begin{subfigure}{0.24\linewidth}
    \includegraphics[width=0.99\linewidth]{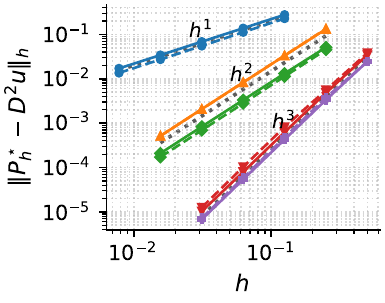}
\end{subfigure}
\begin{subfigure}{0.24\linewidth}
    \includegraphics[width=0.99\linewidth]{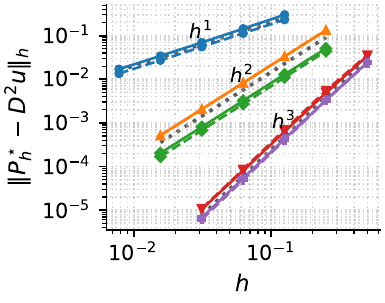}
\end{subfigure}
\begin{subfigure}{0.24\linewidth}
    \includegraphics[width=0.99\linewidth]{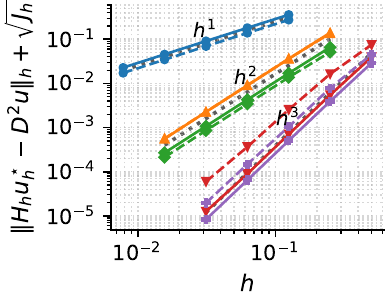}
\end{subfigure}
\begin{subfigure}{0.24\linewidth}
    \includegraphics[width=0.99\linewidth]{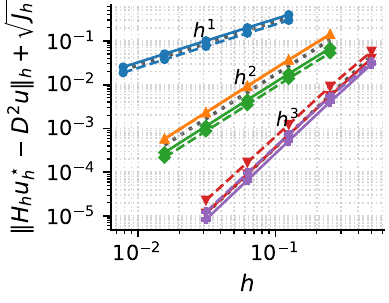}
\end{subfigure}
\begin{subfigure}{0.24\linewidth}
    \includegraphics[width=0.99\linewidth]{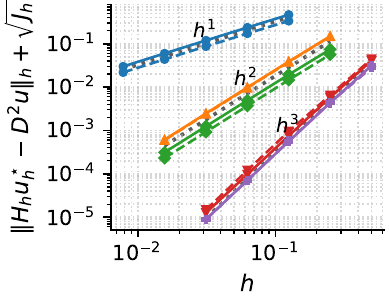}
\end{subfigure}
\begin{subfigure}{0.24\linewidth}
    \includegraphics[width=0.99\linewidth]{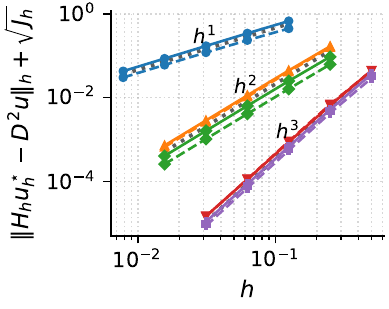}
\end{subfigure}
\vspace{0.2em}
\begin{subfigure}{0.99\linewidth}
    \centering
    \includegraphics[width = \linewidth]{figures/test1_sigma_10/legend_convergence.pdf}
\end{subfigure}
\caption{Columns (from left to right) correspond to $\sigma=1$, $2$, $5$, and $20$. The top row shows $\|\mathbf{P}_h^\star-D^2u\|_h$, while the bottom row shows $\|H_hu_h^\star-D^2u\|_h + J_h(u_h^\star-u,u_h^\star-u)^{1/2}$, both as functions of the mesh size $h$ for test case 1.}
\label{fig:sigma_h}
\end{figure}

\begin{figure}[h]
    \centering

\makebox[\linewidth]{$\sigma = 1$}

\vspace{0.01em}
\begin{subfigure}{0.49\linewidth}
    \centering
    \includegraphics[width=0.48\linewidth]{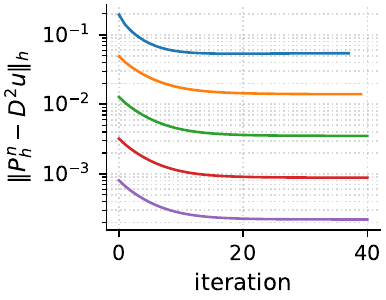}
    \includegraphics[width=0.48\linewidth]{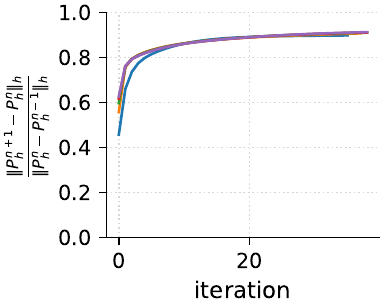}
    \caption*{\centering $C^0$, $m=1$, three-point Gauss rule}
\end{subfigure}
\begin{subfigure}{0.49\linewidth}
    \centering
    \includegraphics[width=0.48\linewidth]{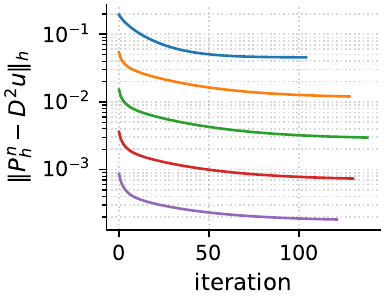}
    \includegraphics[width=0.48\linewidth]{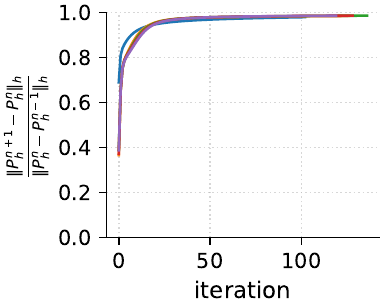}
    \caption*{\centering $DG$, $m=1$, three-point Gauss rule}
\end{subfigure}
\vspace{0.3em}

\makebox[\linewidth]{$\sigma = 2$}
\vspace{0.01em}

\begin{subfigure}{0.49\linewidth}
    \centering
    \includegraphics[width=0.48\linewidth]{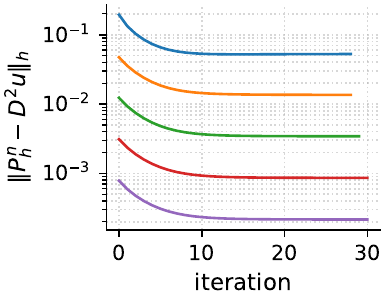}
    \includegraphics[width=0.48\linewidth]{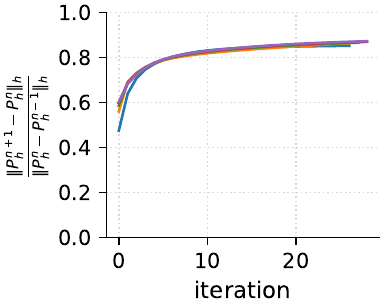}
    \caption*{\centering $C^0$, $m=1$, three-point Gauss rule}
\end{subfigure}
\begin{subfigure}{0.49\linewidth}
    \centering
    \includegraphics[width=0.48\linewidth]{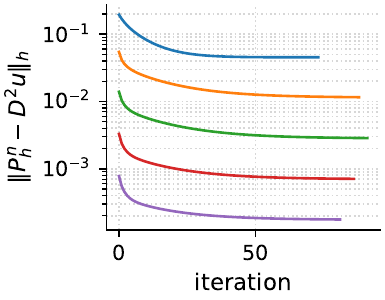}
    \includegraphics[width=0.48\linewidth]{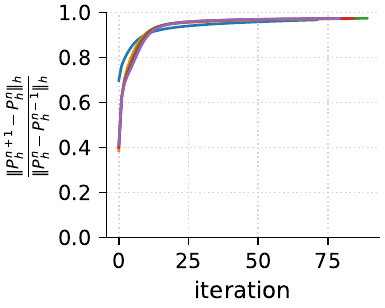}
    \caption*{\centering $DG$, $m=1$, three-point Gauss rule}
\end{subfigure}
\vspace{0.3em}

\makebox[\linewidth]{$\sigma = 5$}
\vspace{0.01em}

\begin{subfigure}{0.49\linewidth}
    \centering
    \includegraphics[width=0.48\linewidth]{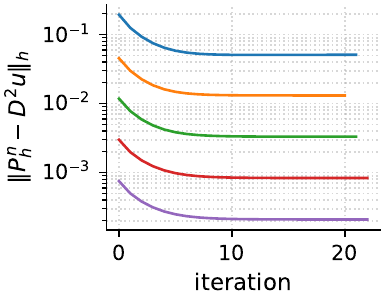}
    \includegraphics[width=0.48\linewidth]{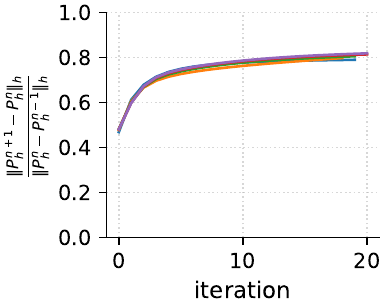}
    \caption*{\centering $C^0$, $m=1$, three-point Gauss rule}
\end{subfigure}
\begin{subfigure}{0.49\linewidth}
    \centering
    \includegraphics[width=0.48\linewidth]{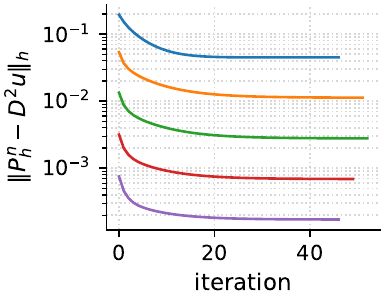}
    \includegraphics[width=0.48\linewidth]{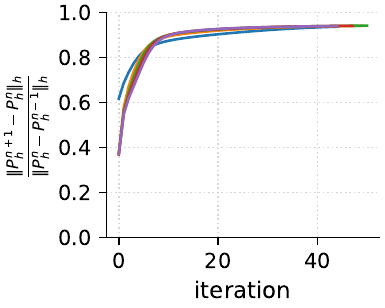}
    \caption*{\centering $DG$, $m=1$, three-point Gauss rule}
\end{subfigure}
\vspace{0.3em}
\makebox[\linewidth]{$\sigma = 20$}
\vspace{0.01em}

\begin{subfigure}{0.49\linewidth}
    \centering
    \includegraphics[width=0.48\linewidth]{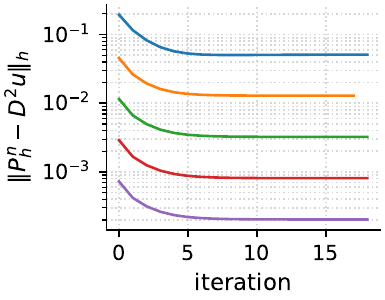}
    \includegraphics[width=0.48\linewidth]{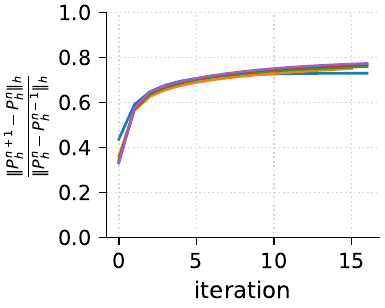}
    \caption*{\centering $C^0$, $m=1$, three-point Gauss rule}
\end{subfigure}
\begin{subfigure}{0.49\linewidth}
    \centering
    \includegraphics[width=0.48\linewidth]{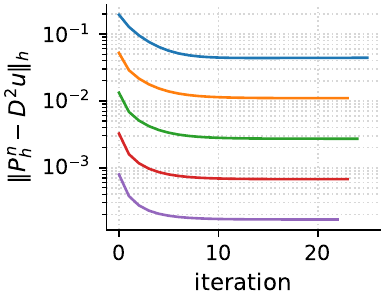}
    \includegraphics[width=0.48\linewidth]{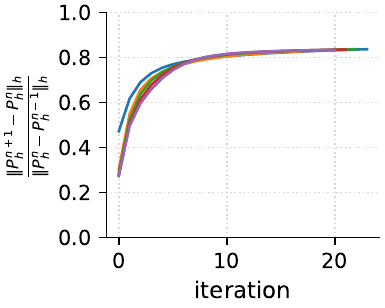}
    \caption*{\centering $DG$, $m=1$, three-point Gauss rule}
\end{subfigure}
\vspace{0.3em}

    \begin{subfigure}{0.99\linewidth}
    \centering
    \includegraphics[width = 0.65\linewidth]{figures/test1_sigma_10/iterations_c0_m1_internal_legend.pdf}
\end{subfigure}
\caption{Error and approximated contraction factor vs. iteration $n$ for $\sigma=1$ (first row), $2$ (second row), $5$ (third row), and $20$ (fourth row). From left to right: the error $\|\mathbf{P}_h^n-D^2u\|_h$ and the ratio  $\|\mathbf{P}_h^{n+1}-\mathbf{P}_h^n\|_h/\|\mathbf{P}_h^{n}-\mathbf{P}_h^{n-1}\|_h$ for $C^0$ finite elements with $m=1$ and the three-point Gauss rule and for DG finite elements with $m=1$ and the three-point Gauss rule.}
\label{fig:sigma_iter}
\end{figure}

\clearpage

\bibliographystyle{alpha}
\bibliography{sample}

\end{document}